\documentclass{amsart}
\usepackage[english]{babel}
\usepackage[latin1]{inputenc}
\usepackage{enumerate}
\usepackage{pdfsync}
\usepackage{color}
\usepackage{amsmath,latexsym,amsfonts,,amssymb}
\usepackage{comment}
\usepackage[showonlyrefs]{mathtools}
\mathtoolsset{showonlyrefs=true}
\newcommand{\R}{{\mathbb R}}
\newcommand{\Z}{{\mathbb Z}}
\newcommand{\N}{{\mathbb N}}

\newcommand {\nin}{\noindent}
\newcommand {\menos}{\backslash}
\newcommand {\gd}{\displaystyle}
\newcommand {\con}{\subset}

\newcommand {\rt}{\rightarrow}

\begin{document}

\title[Local minimizers in spaces of symmetric functions and applications]
{Local minimizers in spaces of symmetric functions and applications}

%----------Author 1

\author{Leonelo Iturriaga}

\address{Leonelo Iturriaga \newline \indent Departamento de Matemática - Universidad Técnica Federico Santa Maria \newline \indent
Av. España 1680, Casilla 110-V, Valparaíso - Chile \newline \indent
leonelo.iturriaga@usm.cl}

\author{Ederson Moreira dos Santos}
\address{Ederson Moreira dos Santos \newline \indent Instituto de Ciências Matemáticas e de Computação - Universidade de São Paulo \newline \indent
C.P. 668, CEP 13560-970 - S\~ao Carlos - SP - Brazil \newline \indent
ederson@icmc.usp.br}

\author{Pedro Ubilla}

\address{Pedro Ubilla \newline \indent Departamento de Matemáticas e C. C. - Universidad de Santiago de Chile \newline \indent
Casilla 307, Correo 2, Santiago - Chile \newline \indent
pedro.ubilla@usach.cl}

\date{\today}



%%% ----------------------------------------------------------------------

\begin{abstract}
We study $H^1$ versus $C^1$ local minimizers for functionals defined on spaces of symmetric functions, namely functions that are invariant by the action of some subgroups of $\mathcal{O}(N)$. These functionals, in many cases, are associated with some elliptic partial differential equations that may have supercritical growth. So we also prove some results on classical regularity for symmetric weak solutions for a general class of semilinear elliptic equations with possibly supercritical growth. We then apply these results to prove the existence of a large number of classical positive symmetric solutions to some concave-convex elliptic equations of Hénon type.
\end{abstract}

\maketitle
\numberwithin{equation}{section}
\newtheorem{theorem}{Theorem}[section]
\newtheorem{lemma}[theorem]{Lemma}
\newtheorem{example}[theorem]{Example}
\newtheorem{remark}[theorem]{Remark}
\newtheorem{proposition}[theorem]{Proposition}
\newtheorem{definition}[theorem]{Definition}
\newtheorem{corollary}[theorem]{Corollary}
\newtheorem*{open}{Open problem}

\section{Introduction}

We study $H^1$ versus $C^1$ local minimizers for functionals defined in spaces of symmetric functions, namely functions that are invariant by the action of some subgroups of $\mathcal{O}(N)$. The functionals considered in this paper may not be defined in the whole space $H^1_0(B)$, but on some proper subspaces of symmetric functions. Throughout in this paper $B$ stands for the open unit ball centered at zero in $\R^N$, $N \geq 1$. In order to prove the equivalence in the $C^1$-topology and $H^1$-topology of local symmetric minimizers, it is essential to have the classical regularity for symmetric weak solutions of the Euler-Lagrange equations associated with these functionals. Problems with supercritical growth in the classical sense are involved and so classical regularity results, as in Brezis and Kato \cite{brezis-kato} based on the Moser's iteration technique \cite{moser}, cannot be directly applied. By the same reason, the principle of symmetric criticality of Palais \cite{palais} does not apply. Hence we prove some regularity results, namely Theorems \ref{teo regularidade radial} and \ref{teo regularidade l}, which cover a large class of elliptic partial differential equations and extend and simplify the proofs of some results in \cite[Sections 5.1 and 5.2]{ederson-djairo-olimpio}. 

%\textcolor{blue}{In addition, they are applied to study steady-state solutions of some reaction-diffusion equations that arise in some astrophysics models with the presence of objects with non-zero dimension that interfere in the reaction rate; cf. \cite[Section 1]{ederson-pacella}.}

We then apply these results to prove the existence of a large number of positive solutions to some classes of elliptic partial differential equations of concave-convex type.  We prove the existence of at least three solutions and, if $N\geq 3$, up to $\left[ \frac{N}{2}\right] + 2$ solutions, each of them exhibiting certain symmetry. In comparison with the pioneering work of Brezis and Nirenberg \cite{brezis-nirenbergHXC} and Ambrosetti et al. \cite{abc}, our approach allows us to obtain the existence of more solutions and to treat problems that are critical or supercritical in the classical sense. 

We consider elliptic equations of the type
\begin{equation}\label{general equation introduction}
 -\Delta u  = f(x, u) \quad \text{in} \quad B, \quad u  = 0 \quad \text{on} \quad \partial B,
\end{equation}
where $f$ satisfies some suitable hypotheses regarding symmetry with respect to the first variable and growth that may even be supercritical in the classical sense. We also assume that $f$ is Caratheodory, that is, for each $u \in \R$, $x \mapsto f(x, u)$ is measurable and, $u \mapsto f(x, u)$ is continuous for almost every $x \in B$.

As we will describe next, many interesting problems involving partial differential equations are invariant by the action of certain groups of symmetries and there are two major lines of research on this type of problems: the symmetry that solutions inherit from the problem, and the existence of solutions exhibiting the problem's symmetry.

On the first direction we mention the seminal work of Gidas et al. \cite{gidas-ni-nirenberg} in which, assuming quite sharp conditions on $f$, radial symmetry for any positive solution of \eqref{general equation introduction} is proved. Bearing on this subject and related to the problems treated in this paper  we also mention the results on symmetry breaking for least energy solutions of the Hénon equation \cite{henon}, i.e. in case $f(x,u) = |x|^\alpha |u|^{p-1}u$ with $\alpha >0$ and $p>1$, proved in \cite{smets-su-willem, byeon-wang, cao-yan-peng} and the results about the Schwarz foliated symmetry for least energy solutions proved in \cite{smets-willem, pacella}.

On the second direction, within which this paper contributes, the search of symmetric solutions naturally induces the study of spaces of symmetric functions. Here we mention the work of Strauss \cite{strauss} on solitary waves; the work of Ni \cite{ni} on the Hénon equation;  the work of Lions \cite{lionssymmetry} about symmetry and compactness on Sobolev spaces; the work of de Figueiredo et al. \cite{ederson-djairo-olimpio} about embeddings of Sobolev spaces of symmetric functions in weighted $L^p$-spaces.

To state the results about  $H^1$ versus $C^1$ local minimizers in spaces of symmetric functions, we introduce some notations: $F(x,u) : = \int_0^u f(x,s) ds$, $\alpha \geq 0$ and
\[
2^* = \left\{
\begin{array}{l}
2N/(N-2) \quad \text{if} \quad N \geq 3,\\
\infty \quad \text{if} \quad N=1,2,
\end{array}
\right.
\quad
2^*_{\alpha} =
\left\{
\begin{array}{l}
2(N+ \alpha)/(N-2) \quad \text{if} \quad N \geq 3, \\
\infty \quad \text{if} \quad N=1,2,
\end{array}
 \right.
\]
which are, in the case of $N\geq 3$, the critical exponents for the embeddings $H^1_0(B) \hookrightarrow L^p(B)$ and $H^1_{0, {\rm rad}}(B) \hookrightarrow L^p(B, |x|^{\alpha})$, respectively.

\begin{theorem}[\textbf{$H^1$ versus $C^1$ local minimizers: space of radially symmetric functions}]\label{teo HXC radial introduction}
Assume the symmetry and growth conditions
\begin{equation}\label{crescimento f regularidade introduction}
\left\{
\begin{array}{l}
f(x,u) = f(|x|, u), \quad \forall \, u \in \R, \ \forall \, x \in B,\\
|f(x, u)| \leq C |x|^{\alpha}(1 + |u|^q), \ \ \forall \, x\in B, \ \ \forall \, u \in \R, \ \ C>0 \ \text{is a constant},\\
\alpha \geq 0, \: q = 2^*_{\alpha}-1 \:\: \text{in case} \:\: N \geq 3 \:\: \text{and any} \:\: 1 < q \:\: \text{in case}  \:\:N=1,2,
\end{array}
\right.
\end{equation}
and set
\[
\begin{array}{l}
H^1_{0, {\rm rad}}(B) = \{ u \in H^1_0(B); u = u \circ O,  \ \forall \ O \in \mathcal{O}(N) \} \quad \text{and} \\
C^1_{0, {\rm rad}}(\overline{B}) = \{ u \in C^1(\overline{B}); u = u \circ O, \  \forall \  O \in \mathcal{O}(N) \  \text{and}  \ u = 0  \ \text{on} \ \partial B\}.
\end{array}
\]
Associated with \eqref{general equation introduction} we consider the integral functional
\begin{equation}\label{funcional geral regularidade}
\Phi_{{\rm rad}}(u) = \frac{1}{2} \int_B|\nabla u|^2 dx - \int_B F(x,u)\, dx, \quad u \in H^1_{0,{\rm rad}}(B).
\end{equation}
Let $u_0 \in H^1_{0, {\rm rad}}(B)$ be a local minimum of $\Phi_{{\rm rad}}$ for the $C^1_{{\rm rad}}$-topology, that is, there exists $r> 0$ such that
 \begin{equation}\label{minimum C1 introduction}
 \Phi_{{\rm rad}}(u_0) \leq \Phi_{{\rm rad}}(u_0 + v), \quad \forall \,  v \in C^1_{0, {\rm rad}}(\overline{B}) \quad \text{with} \quad \| v \|_{C^1} \leq r.
 \end{equation}
 Then $u_0$ is also a local minimum of $\Phi_{{\rm rad}}$ for the $H^1_{{\rm rad}}$-topology, that is, there exists $\delta> 0$ such that
 \begin{equation}\label{minimum H1 introduction}
 \Phi_{{\rm rad}}(u_0) \leq \Phi_{{\rm rad}}(u_0 + v), \quad \forall \, v \in H^1_{0, {\rm rad}}(B) \quad \text{with} \quad \| \nabla v\|_{2} \leq \delta.
 \end{equation}
\end{theorem}

\vspace{10pt}

\begin{theorem}[\textbf{$H^1$ versus $C^1$ local minimizers: spaces of partially \linebreak symmetric functions}]\label{HXC partial introduction}
Let $N\geq 4$, $l \in \N$ such that  $2 \leq N- l \leq l$. Set $2_ l = \frac{2(l+1)}{l-1}$,
\[
\begin{array}{l}
H_{l}(B)  =  \{ u \in H^1_0(B); u = u \circ O, \  \forall \ O \in \mathcal{O}(l) \times \mathcal{O}(N-l) \} \quad \text{and} \\
C^1_{0, l}(\overline{B}) = \{ u \in C^1(\overline{B}); u = u \circ O, \ \forall \ O \in \mathcal{O}(l) \times \mathcal{O}(N-l) \ \text{and} \ u = 0 \ \text{on} \ \partial B\}.
\end{array}
\]
Assume the symmetry and growth conditions
\begin{equation}\label{crescimento f l introduction}
\left\{
\begin{array}{l}
f(y,z,u) = f(|y|, |z|,u), \quad \forall \, x = (y,z) \in B, \ \forall \, u \in \R,\\
|f(x, u)| \leq C |x|^{\alpha}(1 + |u|^{p}), \ \ \forall \, x\in B, \ \ \forall \, u \in \R, \ \ C> 0 \: \text{is a constant}, \\
\alpha \geq \alpha_0(N, l, p) >0, \quad 1 < p < 2_l -1 = \frac{l+3}{l-1},
\end{array}
\right.
\end{equation}
where $\alpha_0(N,l,p)$ is given as in Theorem \rm{\ref{teo regularidade l}} below. Associated with \eqref{general equation introduction} we consider the integral functional
\begin{equation}\label{funcional geral l introduction}
 \Phi_{l}(u) = \frac{1}{2} \int_B|\nabla u|^2 dx - \int_B F(x,u)\, dx, \quad u \in H_{l}(B).
\end{equation}
Let $u_0 \in H_l$ be a local minimum of $\Phi_l$ for the $C^1_l$-topology, that is, there exists $r> 0$ such that
\[
 \Phi_l(u_0) \leq \Phi_{l}(u_0 + v), \quad \forall \,  v \in C^1_{0, l}(\overline{B}) \quad \text{with} \quad \| v \|_{C^1} \leq r.
\]
 Then $u_0$ is also a local minimum of $\Phi_{l}$ for the $H_l$-topology, that is, there exists $\delta> 0$ such that
\[
 \Phi_{l}(u_0) \leq \Phi_{l}(u_0 + v), \quad \forall \, v \in H_l(B) \quad \text{with} \quad \| \nabla u\|_{2} \leq \delta.
\]
\end{theorem}

\vspace{10pt}

A model problem to which the above results apply is
\begin{equation}\label{problem weight introduction}
\left\{
\begin{array}{l}
-\Delta w = \lambda | x |^{\alpha} (w + 1 )^{p} \quad \text{in} \quad B, \\
\quad w > 0 \quad \text{in} \quad B, \quad w = 0 \quad \text{on} \quad \partial B,
\end{array}
\right.
\end{equation}
with $\alpha > 0$, $p>1$ and a parameter $\lambda> 0$, which is related to the Hénon equation
\[
-\Delta w = | x |^{\alpha}w^p \quad \text{in} \quad B, \\
\quad w > 0 \quad \text{in} \quad B, \quad w = 0 \quad \text{on} \quad \partial B,
\]
and to the equation
\begin{equation}\label{equation non weight}
-\Delta w = \lambda (w + 1 )^{p} \quad \text{in} \quad B, \\
\quad w > 0 \quad \text{in} \quad B, \quad w = 0 \quad \text{on} \quad \partial B,
\end{equation}
for which we cite the works \cite{keener-keller, joseph-lundgren, crandall-rabinowitz, brezis-nirenberg, gazzola-malchiodi}. 

The equation \eqref{equation non weight} has been extensively studied due to its application to physical models, in particular as the  steady-state problem corresponding to a nonlinear reaction-diffusion equation;  cf. \cite{joseph-lundgren, joseph-sparrow}. By adding to \eqref{equation non weight} the weight $|x|^\alpha$, with $\alpha> 0$, which turns out to be the equation \eqref{problem weight introduction}, we mean that the medium $B$ has some intrinsic properties that interfere in the reaction rate. Moreover, the partial differential equation \eqref{problem weight introduction}, literally the identity \eqref{problem weight introduction}, says that such intrinsic properties of $B$ hinder diffusion close to its center, because $|x|^{\alpha}$ vanishes at zero and it is very small close to $x=0$. Therefore, the existence of solutions that concentrate on the boundary, as $\alpha \rt + \infty$, is somehow expected; cf. \cite{byeon-wang, cao-yan-peng, ederson-pacella}.

In view of the published literature on the non-weighted problem \eqref{equation non weight}, e.g. \cite{keener-keller, crandall-rabinowitz, brezis-nirenberg}, the next theorem is quite conventional as it describes the range of the parameter $\lambda$ in which \eqref{problem weight introduction} has a solution.

\begin{theorem}[\textbf{existence of a solution}]\label{teorema existencia introduction}
Suppose $N \geq 1$, $p>1$ and $\alpha > 0$. There exists $\lambda_* = \lambda_*(N, \alpha, p) > 0$ such that:
\begin{enumerate}[{\rm(i)}]
\item  There is no classical solution of \eqref{problem weight introduction} if $\lambda > \lambda_*$;

\item  There exists at least one radial classical solution of \eqref{problem weight introduction} if $0< \lambda < \lambda_*$;

\item If $N \geq 3$ also assume $p \leq 2^*_{\alpha} -1$. If $ \lambda = \lambda_*$, then \eqref{problem weight introduction} has at least one radial classical solution.
\end{enumerate}
\end{theorem}

We then combine Theorems \ref{teo HXC radial introduction}, \ref{HXC partial introduction} and \ref{teorema existencia introduction} to guarantee the existence of a large number solutions to \eqref{problem weight introduction} and, in addition, we classify their symmetry.

\begin{theorem}[\textbf{multiple symmetric solutions}]\label{teorema multiplicidade introduction}
Suppose $N \geq 1$ and $\alpha > 0$. Let $\lambda_*>0$ be as above.
\begin{enumerate}[{\rm (I)}]
 \item Let $1 < p$ and if $N \geq 3$ also assume $p \leq 2^*_{\alpha} -1$. If $0< \lambda < \lambda_*$, then \eqref{problem weight introduction} has at least two radial classical solutions.
\item Let $1< p$ and $N=1,2$.  There exists $\alpha_0 = \alpha_0 (N,p) > 0$ such that the problem \eqref{problem weight introduction} has at least three non rotational equivalent classical solutions, if $\alpha > \alpha_0$ and $0< \lambda < \lambda_0(N,p,\alpha)$. Two of them are radially symmetric. If $N=1$, the third solution is not even. If $N=2$, the third solution is not radially symmetric but Schwarz foliated symmetric.
\item Let $N\geq 3$ and  $1< p < 2^*-1$. There exists $\alpha_0 = \alpha_0 (N,p) > 0$ such that the problem \eqref{problem weight introduction} has at least $\left[\frac{N}{2}\right] + 2$ non rotational equivalent classical solutions, if $\alpha > \alpha_0$ and $0< \lambda < \lambda_0(N,p,\alpha)$.  Two of them are radially symmetric. The third solution is not radially symmetric but Schwarz foliated symmetric. Each of the others $\left[ \frac{N}{2} \right] - 1$ solutions has a $\mathcal{O}(l) \times \mathcal{O}(N- l)$ symmetry for some $l \in \N$ such that $2 \leq N -  l \leq l$.
\end{enumerate}
\end{theorem}

\begin{remark}\label{remark introduction more solutions}
In the case of $N \geq 4$ we have more information about the existence multiple solutions to \eqref{problem weight introduction}. Indeed, we prove the existence of more than two solutions even with $p+1 \geq 2^*$; cf. Remark \ref{remark multiple solutions supercritical}. The main ingredients in the proof are the results about $H^1$ versus $C^1$ minimizers applied to the space $H^1_{0, {\rm rad}}(B)$ combined with a careful asymptotic analysis, as $\lambda \rt 0^+$, of the associated symmetric mountain pass levels.
\end{remark}

In contrast, in the non-weighted case, with $1 < p$ if $N=1,2$, and $1< p\leq 2^*-1$ if $N \geq 3$,  it is proved in \cite{joseph-lundgren} that \eqref{equation non weight} has precisely two solutions for every $0 < \lambda < \lambda_{*}$. 

Our results also apply to the Hénon equation with inhomogeneous boundary condition, namely to
\begin{equation}\label{boundary inhomogeneous}
-\Delta u = | x |^{\alpha} u^p\quad \text{in} \quad B, \quad u > 0 \quad \text{in} \quad B,  \quad \text{with} \quad u = a \quad \text{on} \quad \partial B.
\end{equation}
Indeed, if we write $u = a (w+1)$, then \eqref{boundary inhomogeneous} is equivalent to \eqref{problem weight introduction} with $a^{p-1} = \lambda$. So, Theorems \ref{teorema existencia introduction} and \ref{teorema multiplicidade introduction} can be restated in the context of \eqref{boundary inhomogeneous}.

Since the works of Hénon \cite{henon} and Ni \cite{ni}, the problem \eqref{boundary inhomogeneous} with $a =0$ has been extensively studied; cf. \cite{smets-su-willem, pacella, byeon-wang, badiale-serra, pistoia-serra, ederson-djairo-olimpio, cao-peng, smets-willem, wei, GGN} and references therein. Neverthless, less attention has been devoted to the study of \eqref{boundary inhomogeneous} with $a > 0$ and, as far as we know, \cite{phan-souplet} is the only published paper on this subject. 
%It is known that the Hénon equation \eqref{boundary inhomogeneous} corresponds to a model for steady state solutions for distribution of stars in the globular cluster $B$ and $u$ represents the density of light in this cluster. So, in particular, Theorem \ref{teorema existencia introduction} says that there is a superior limit of densities of light on the border of the cluster $\partial B$ for steady state radial solutions.

One of the first applications of the result of Brezis and Nirenberg \cite{brezis-nirenbergHXC} about $C^1$ versus $H^1$ local minimizers appears in concave-convex problems as in Ambrosetti et al. \cite{abc}, whose  weighted version reads
\begin{equation}\label{concave-convex}
\left\{
\begin{array}{l}
-\Delta u  = \lambda |x|^\beta u^q + |x|^{\alpha} u^p \quad \text{in} \quad B,\\
u > 0 \quad \text{in} \quad B, \quad u = 0 \quad \text{on} \quad \partial B,
\end{array}
\right.
\end{equation}
with $\beta \geq 0$, $\alpha \geq 0$ and $0 < q < 1 < p \leq 2^*_\alpha -1$; cf. \cite{clement-djairo-mitidieri} where \eqref{concave-convex} is considered in the case of $q=1$. We stress that Theorems \ref{teo HXC radial introduction} and \ref{HXC partial introduction} can be applied to study equation \eqref{concave-convex} and to obtain results in the sense of Theorems \ref{teorema existencia introduction} and \ref{teorema multiplicidade introduction} above.

Our results also apply to some weighted problems posed in exterior domains, as to
\begin{equation}\label{exterior introduction}
 \left\{
\begin{array}{l}
 -\Delta U = \gd{\frac{U^p}{|x|^{\beta}}} \quad \text{in} \quad \R^N \menos \overline{B}, \ \ \ \ U  >  0 \quad \text{in} \quad \R^N \menos \overline{B},\\
 U  =  a \quad \text{on} \quad \partial B, \ \ U  \rt 0 \quad \text{as} \quad |x| \rt \infty,
\end{array}
 \right.
\end{equation}
with $N \geq 3$, $a\geq0$, $\beta \in \R$ and $p>0$; cf. Section \ref{section related problems}. In this direction our results slightly complement the results in \cite[Theorem 2]{davilla}. 

This manuscript is organized as follows. In Section \ref{regularity results} we prove some regularity results for symmetric weak solutions of \eqref{general equation introduction}, namely Theorems \ref{teo regularidade radial} and \ref{teo regularidade l}. In Sections \ref{H1XC1 radial}, \ref{section existencia} and \ref{section multiplicidade} we prove Theorems \ref{teo HXC radial introduction}, \ref{teorema existencia introduction} and \ref{teorema multiplicidade introduction}, respectively. Then, in Section \ref{section related problems}, we describe how our procedure can be used to prove results on the existence and the multiplicity of solutions to the equation \eqref{exterior introduction}.

\section{Regularity results} \label{regularity results}
In this section we prove classical regularity for symmetric weak solutions of elliptic equations of the type \eqref{general equation introduction}. We point out that problems with supercritical growth in the classical sense are involved and so classical regularity results, as in Brezis and Kato \cite{brezis-kato} based on the Moser's iteration technique \cite{moser}, cannot be directly applied. In addition, since the functionals may not be defined on $H^1_0(B)$, the principle of symmetric criticality \cite{palais} cannot be applied as well.

\subsection{Radial solutions}
We recall that from \cite[eq. (4)]{ni}, for $N\geq 3$ and any $u \in H^1_{0, {\rm rad}}(B)$ we have
\begin{equation}\label{des ni regularidade}
|u(x)| \leq \frac{1}{\sqrt{\omega_N (N-2)}} \frac{\| \nabla u\|_2}{|x|^{\frac{N-2}{2}}}, \quad \forall \, x \in B\menos\{0\},
\end{equation}
where $\omega_N$ stands for the surface area of $S^{N-1} \con \R^N$.

Then, \eqref{crescimento f regularidade introduction}, \eqref{des ni regularidade} for $N\geq 3$ and the classical Sobolev embeddings of $H^1_0(B)$ for $N=1,2$ guarantee that
\[
 \Phi_{{\rm rad}}(u) = \frac{1}{2} \int_B|\nabla u|^2 dx - \int_B F(x,u)\, dx, \quad u \in H^1_{0,{\rm rad}}(B),
\]
is a $C^1(H^1_{0, {\rm rad}}(B), \R)$-functional.

\begin{definition}
We say that $u \in H^1_{0, {\rm rad}}(B)$ is a weak radial solution of \eqref{general equation introduction} if $\Phi_{{\rm rad}}'(u) = 0$, that is, if
\[
\int_B \nabla u \nabla v dx = \int_B f(x,u)v dx, \quad \forall \, v \in H^1_{0,{\rm rad}}(B).
\]
\end{definition}

\begin{theorem}\label{teo regularidade radial}
Assume \eqref{crescimento f regularidade introduction}. If $u \in H^1_{0, {\rm rad}}(B)$ is a weak radial solution of \eqref{general equation introduction}, then $u \in W^{2,t}(B) \cap W^{1,t}_0 (B)$ for every $t \geq 1$ and it strongly solves \eqref{general equation introduction}. Consequently, by classical embeddings of Sobolev spaces, $u \in C^{1, \theta}(\overline{B})$ for every $0 < \theta < 1$.
\end{theorem}
\begin{proof}
Let $u \in H^1_{0,{\rm rad}}(B)$ be a weak solution of \eqref{general equation introduction}. Then, by \eqref{crescimento f regularidade introduction}, $f(x,u) \in L^t(B)$ for all $t \geq 1$ in case $N=1,2$.  On the other hand, in case $N \geq 3$, from \eqref{crescimento f regularidade introduction} we infer that
\begin{multline}\label{conta fundamental radial}
|f(x,u)|^{\frac{2N}{N+2}} \leq (C |x|^{\alpha}(1 + |u|^{\frac{N+2+ 2\alpha}{N-2}}))^\frac{2N}{N+2}\\ \leq C ( |x|^{\frac{2N\alpha}{N+2}} + |u|^{\frac{2N}{N-2}} |x|^{\frac{2N\alpha}{N+2}}|u|^{\frac{4N\alpha}{(N-2)(N+2)}}),
\end{multline}
and, by \eqref{des ni regularidade}, $|x|^{\frac{2N\alpha}{N+2}}|u|^{\frac{4N\alpha}{(N-2)(N+2)}} \in L^{\infty}(B)$. Then $f(x,u) \in L^{\frac{2N}{N+2}}(B)$ in case $N \geq 3$.

Let $w$ be the strong solution of
\[
-\Delta w = f(x,u) \quad \hbox{in} \quad B, \quad \hbox{with} \quad w = 0 \quad \hbox{on} \quad \partial B.
\]
Then, by the standard elliptic regularity for second order elliptic operators as in \cite{agmon-douglis-nirenberg}, $w \in W^{2,t} (B) \cap W^{1,t}_0(B)$ for all $t\geq 1$ in case $N=1,2$ and $w \in W^{2,\frac{2N}{N+2}} (B) \cap W^{1,\frac{2N}{N+2}}_0(B)$ in case $N \geq 3$. In addition, $w$ is radially symmetric. Then, for every $\varphi \in C_{0, {\rm rad}}^{\infty} (\overline{B})$,
\begin{equation}\label{partes-radial-Ni}
\int_B w (-\Delta \varphi) dx = \int_B (-\Delta w) \varphi dx = \int_B f(x,u)\varphi dx  = \int_B \nabla u \nabla \varphi = \int_B u (-\Delta \varphi) dx.
\end{equation}
Let $\psi \in C_{c, {\rm rad}}^{\infty} (B)$ be an arbitrary function and $\varphi$ be the solution of
\[
-\Delta \varphi = \psi \quad \hbox{in} \quad B, \quad \hbox{with} \quad \varphi =0 \quad \hbox{on} \quad \partial B.
\]
Then $\varphi \in C_{0, {\rm rad}}^{\infty}(\overline{B})$ and so, by means of \eqref{partes-radial-Ni},
\[
\int_{B} (w -u) \psi dx = 0 \quad \forall \, \psi \in C_{c, {\rm rad}}^{\infty} (B),
\]
which implies
\[
\int_0^1 [w(r)-u(r)]\psi(r) r^{N-1} dr = 0 \quad \forall \, \psi \in C^{\infty}_c((0,1)).
\]
Therefore, $w =u$ a.e. in $B$.

At this point we have proved that $u$ strongly solves \eqref{general equation introduction} and that $u \in W^{2,t} (B) \cap W^{1,t}_0(B)$ for all $t\geq 1$ in case $N=1,2$ and $u \in W^{2,\frac{2N}{N+2}} (B) \cap W^{1,\frac{2N}{N+2}}_0(B)$ in case $N \geq 3$.

In case $N\geq 3$, since $W^{2,\frac{2N}{N+2}}(B) \hookrightarrow H^{1}_0(B)$ and since $u \in W^{2,\frac{2N}{N+2}} (B) \cap W^{1,\frac{2N}{N+2}}_0(B)$ strongly solves \eqref{general equation introduction} we also obtain that $u$ weakly solves \eqref{general equation introduction} in the sense of $H^1_0(B)$. In addition, by \eqref{crescimento f regularidade introduction} we write
\begin{equation}\label{legal des regularidade}
|f(x,u)| \leq C |x|^{\alpha}(1 + |u|^{\frac{N + 2 + 2\alpha}{N-2}})  = a(x)(1 + |u|), \quad \forall \, x \in B,
\end{equation}
with
\[
a(x) = C |x|^{\alpha}\dfrac{1 + |u|^{\frac{N + 2 + 2\alpha}{N-2}}}{1 + |u|}.
\]
Then
\[
\begin{array}{rcl}
(a(x))^{\frac{N}{2}} &\leq& \left(C |x|^{\alpha} (1 + |u|^{\frac{2(2 + \alpha)}{N-2}})\right)^{\frac{N}{2}} \leq C_1\left( |x|^{\frac{\alpha N}{2}} +  |x|^{\frac{\alpha N}{2}}|u|^{\frac{(2+\alpha)N}{N-2}} \right)\vspace{2pt}\\
&=& C_1\left( |x|^{\frac{\alpha N}{2}} +  |x|^{\frac{\alpha N}{2}}|u|^{\frac{\alpha N}{N-2}} |u|^{\frac{2N}{N-2}} \right)
\end{array}
\]
and by Ni's pointwise estime \eqref{des ni regularidade} we infer that $|x|^{\frac{\alpha N}{2}}|u|^{\frac{\alpha N}{N-2}} \in L^{\infty}(B)$. Therefore, $a(x) \in L^{\frac{N}{2}}(B)$. Then, by \cite[Theorem 2.3]{brezis-kato}, see also \cite[B.3 Lemma]{struwe}, and \cite[Lemma 9.17]{gilbarg-trudinger}, it follows that $u \in W^{2,t}(B)\cap W^{1,t}_0(B)$ for all $t\geq 1$ and strongly solves \eqref{general equation introduction}. 
\end{proof}

\begin{remark}\label{classical regularity radial}
If $f(x,u)$ satisfies \eqref{crescimento f regularidade introduction} and is also suitably regular, that is the case of $f(x,u) = \lambda |x|^{\alpha}|u+1|^{p-1}(u+1)$ with $\lambda \in \R$, $\alpha \geq 0$ and $p >1$, then we combine Theorem {\rm \ref{teo regularidade radial}} and the classical Schauder's estimates to obtain that any weak radial solution of \eqref{general equation introduction} lies in $C^{2, \gamma}(\overline{B})$, for a certain $0 < \gamma < 1$, and classically solves \eqref{general equation introduction}.
\end{remark}

\subsection{Partially symmetric solutions}
In this part
\begin{equation}\label{condition partial symmetry}
N\geq 4, \ \ l \in \N \ \ \text{s.t.} \ \ 2 \leq N- l \leq l, \ \ \alpha > \max\left\{\frac{(N+ 2)^2}{2(N-2)}, \frac{(N-2)^2 + 5}{ N-3}\right\}.
\end{equation}
Set $2_ l = \frac{2(l+1)}{l-1}$ and
\[
\begin{array}{rcl}
H_{l}(B) & = & \{ u \in H^1_0(B); u = u \circ O \ \ \text{for every} \ \ O \in \mathcal{O}(l) \times \mathcal{O}(N-l) \}\\
&= & \{ u \in H^1_{0}(B); u(y,z) = u(|y|, |z|), \,x= (y,z) \in B \},
\end{array}
\]
where $y = (y_1, \ldots, y_l)$,  $z = (z_1, \ldots, z_{N-l})$. Then, as proved in \cite[Theorem 2.1]{badiale-serra}, we have $H_{l}(B) \hookrightarrow L^{2_l}(B, |x|^{\alpha})$ continuously. Observe that instead of $\alpha > N+2$, the inequality $\alpha > \max\left\{\frac{(N+ 2)^2}{2(N-2)}, \frac{(N-2)^2 + 5}{ N-3}\right\}$ is the correct assumption for the proof of \cite[Theorem 2.1]{badiale-serra}.

Here we assume \eqref{condition partial symmetry} and  the symmetry and growth conditions \eqref{crescimento f l introduction}. Under these hypotheses
\[
 \Phi_{l}(u) = \frac{1}{2} \int_B|\nabla u|^2 dx - \int_B F(x,u)\, dx, \quad u \in H_{l}(B),
\]
is a $C^1(H_{l}(B), \R)$-functional.

\begin{definition}
We say that $u \in H_{l}(B)$ is a weak $(l, N-l)$-symmetric solution of \eqref{general equation introduction} if $\Phi_{l}'(u) = 0$, that is, if
\[
\int_B \nabla u \nabla v dx = \int_B f(x,u)v dx, \quad \forall \, v \in H_{l}(B).
\]
\end{definition}

\begin{theorem}\label{teo regularidade l}
Assume $N\geq 4$, $l \in \N$, $2 \leq N- l \leq l$ and that \eqref{crescimento f l introduction} holds. Then there exists $\alpha_0 = \alpha_0(N, l, p) >0$ with the following property: if $\alpha \geq \alpha_0(N,l,p)$ and $u \in H_{l}(B)$ is a weak $(l,N-l)$-symmetric solution of \eqref{general equation introduction}, then $u \in W^{2,t}(B) \cap W^{1,t}_0 (B)$ for every $t \geq 1$ and it strongly solves \eqref{general equation introduction}. Consequently, by classical embeddings of Sobolev spaces, $u \in C^{1, \gamma}(\overline{B})$ for every $0 < \gamma < 1$.
\end{theorem}
\begin{proof}
By \eqref{crescimento f l introduction} and the imbedding $H_{l}(B) \hookrightarrow L^{2_l}(B, |x|^{\alpha})$, it follows that $f(x,u) \in L^{\frac{2_l}{p}}(B)$. Let $w$ be the strong solution of
\[
-\Delta w = f(x,u) \quad \text{in} \quad B, \quad \text{with} \quad w = 0 \quad \text{on} \quad \partial B.
\]
Then, by the standard elliptic regularity for second order elliptic operators as in \cite{agmon-douglis-nirenberg}, $w \in W^{2,\frac{2_l}{p}}_l (B) \cap W^{1,\frac{2_l}{p}}_0(B)$. Then, for every $\varphi \in C_{0, l}^{\infty} (\overline{B})$,
\begin{equation}\label{partes-badiale-serra}
\int_B w (-\Delta \varphi) dx = \int_B (-\Delta w) \varphi dx = \int_B f(x,u) \varphi dx  = \int_B \nabla u \nabla \varphi = \int_B u (-\Delta \varphi) dx.
\end{equation}
Given $\psi \in C_{c, l}^{\infty} (B)$ let $\varphi$ be the solution of
\[
-\Delta \varphi = \psi \quad \hbox{in} \quad B, \quad \hbox{with} \quad \varphi =0 \quad \hbox{on} \quad \partial B.
\]
Then $\varphi \in C_{0, l}^{\infty}(\overline{B})$ and so, by means of \eqref{partes-badiale-serra},
\[
\int_{B} (w -u) \psi dx = 0 \quad \forall \, \psi \in C_{c, l}^{\infty} (B),
\]
which implies
\[
\int_0^1\int_0^{\sqrt{1- t^2}} [w(s,t)-u(s,t)]\psi(s,t) s^{l-1}t^{N-l-1} dsdt = 0 \quad  \forall \, \psi \in C^{\infty}_c(\mathcal{Q}),
\]
where $\mathcal{Q} = \{(s,t) \con \R^2; s,t> 0 \: \text{and} \: 0 < s^2 + t^2 < 1 \}$.
Therefore, $w =u$ a.e. in $B$.

At this point we have proved that $u \in W^{2,\frac{2_l}{p}}_l (B) \cap W^{1,\frac{2_l}{p}}_0(B)$ and strongly solves \eqref{general equation introduction}.

If $p(l-1) - 4 \leq 0$ and $\alpha$ is large enough, then $u \in W^{2,\theta}(B)$ for all $\theta \geq 1$ as a consequence of \cite[Corollary 1.5]{ederson-djairo-olimpio} with $m =2$ and the standard regularity for second order elliptic operators.

In case $p(l-1) -4 > 0$, \cite[Corollary 1.5]{ederson-djairo-olimpio} is the key in the bootstrap argument. To show that $u \in W^{2,\theta} (B)$ for all $\theta \geq 1$ when $\alpha$ is large enough, we combine \cite[Corollary 1.5]{ederson-djairo-olimpio} with $m=2$ and the standard regularity for second order elliptic operators to prove the statement:
\[
\hbox{if } \  u \in W^{2,q}_{l} (B) \ \hbox{for some} \ q \geq \frac{2_l}{p}, \ \hbox{then} \ u\in W^{2,\theta} (B) \ \hbox{with} \ \frac{\theta}{q} \geq \frac{l-1}{p (l-1) -4}.
\]
And the bootstrap argument works because $\frac{l-1}{p (l-1) -4} >1$ is sharply guaranteed by the hypothesis $p <2_l -1$. 
\end{proof}

\begin{remark}\label{remark regularity l}
If $f(x,u)$ satisfies \eqref{crescimento f l introduction} and is also suitably regular, that is the case of $f(x,u) = \lambda |x|^{\alpha}|u+1|^{p-1}(u+1)$ with $\lambda \in \R$, $\alpha \geq 0$ and $p >1$, then we combine Theorem {\rm \ref{teo regularidade l}} and the classical Schauder's estimates to obtain that any weak $(l,N-l)$-symmetric solution of \eqref{general equation introduction} lies in $C^{2, \gamma}(\overline{B})$, for a certain $0 < \gamma < 1$, and classically solves \eqref{general equation introduction}.
\end{remark}

\section{$H^1$ versus $C^1$ local minimizers: radially symmetric functions} \label{H1XC1 radial} In this part we consider the functional
\[
 \Phi_{{\rm rad}}(u) = \frac{1}{2} \int_B|\nabla u|^2 dx - \int_B F(x,u)\, dx, \quad u \in H^1_{0,{\rm rad}}(B).
\]
Here we suppose that $f$ verifies the growth and symmetry conditions \eqref{crescimento f regularidade introduction} and we prove Theorem \ref{teo HXC radial introduction}. We also mention \cite[Proposition 3.9]{djairo-gossez-ubilla2009} and \cite[Lemma 2.2]{brock-iturriaga-ubilla} from where we borrow some ideas.

\begin{proof}[\textbf{Proof of Theorem {\rm \ref{teo HXC radial introduction}}}] ${}$

\nin Step 1. We infer that $u_0 \in W^{2,t}(B) \cap W^{1,t}_0 (B)$ for every $t \geq 1$ and strongly solves \eqref{general equation introduction}. In particular, $u_0 \in C^{1, \gamma}(\overline{B})$ for every $0 < \gamma < 1$.

Let $(\rho_n)$ a regularizing sequence formed by radial functions, for example,
\[
\rho(x)  = \left\{
\begin{array}{l}
 \exp\left( \frac{1}{|x|^2 -1} \right), \: |x| < 1,\\
 0, \:|x| \geq 1,
\end{array}
\right.
\ \text{and} \ \rho_n(x) = \left( \int_{\R^N} \rho(x) \, dx  \right)n^N \rho(nx), n \in \N.
\]
Then observe that $\rho_n * w$ is radially symmetric if $w$ is radially symmetric. Then $C_{c, {\rm rad}}^{\infty}(B)$, and consequently $C^1_{0, {\rm rad}}(\overline{B})$, is dense in $H^1_{0, {\rm rad}}(B)$.  As in \eqref{conta fundamental radial}, we have that $f(x, u_0) \in L^{\frac{2N}{N+2}}(B)$ in case $N \geq 3$, and $f(x, u_0) \in L^t(B)$ for all $t \geq 1$ in case $N=1,2$.

 From \eqref{minimum C1 introduction}, we have
\[
\int_B \nabla u_0 \nabla v dx - \int_B f(x, u_0)v dx  = 0 \quad \forall \, v \in C^1_{0, {\rm rad}}(\overline{B}),
\]
and since $C^1_{0, {\rm rad}}(\overline{B})$ is dense in $H^1_{0, {\rm rad}}(B)$ and the above integrability of $f(x,u_0)$,
\[
\int_B \nabla u_0 \nabla v dx - \int_B f(x, u_0)v dx  = 0 \quad \forall \, v \in H^1_{0, {\rm rad}}(B),
\]
that is, $u_0$ is a critical point of $\Phi_{{\rm rad}}$. Therefore, by Theorem \ref{teo regularidade radial}, $u_0 \in W^{2,t}(B) \cap W^{1,t}_0 (B)$ for every $t \geq 1$ and strongly solves \eqref{general equation introduction}. In particular, $u_0 \in C^{1, \gamma}(\overline{B})$ for every $0 < \gamma < 1$.
\medbreak

\nin Step 2. We infer that $u_0$ satisfies \eqref{minimum H1 introduction}.

For each $j \in \N$ consider the truncated functional
\[
\Phi_j(u) = \frac{1}{2}\int_B |\nabla u|^2 dx - \int_B F_j(x, u)dx, \quad u \in H^1_{0, {\rm rad}}(B),
\]
where $F_j(x,u) = \int_0^u f_j(x,s)ds$ and
\[
f_j(x,s) = \left\{
\begin{array}{l}
 f(x,-j), \quad \text{if} \quad s < -j,\\
 f(x,s), \quad \text{if} \quad -j \leq s \leq j,\\
 f(x,j), \quad \text{if} \quad s >j.
\end{array}
\right.
\]
Then from \eqref{crescimento f regularidade introduction}, there exists $C> 0$, that does not depend on $j$, such that
\[
|F_j(x,u)| \leq C |x|^{\alpha}(1 + |u|^{q+1}), \quad \forall \, u \in \R, \, \forall \, x \in B,
\]
for certain $q> 1$ in case $N=1,2$ and $q = \frac{N+2(1+\alpha)}{N-2}$ in case $N\geq3$. Also observe that, by the dominated convergence theorem, $\Phi_j(u) \rt \Phi(u)$, as $j \rt \infty$, for every $u \in H^1_{0, {\rm rad}}(B)$.

Suppose that \eqref{minimum H1 introduction} does not hold. Then, by the continuous embedding \linebreak $H^1_{0, {\rm rad}}(B) \hookrightarrow L^{q+1}(B, |x|^{\alpha})$, for each $n \in \N$ there exists $v_{n} \in H^1_{0, {\rm rad}}(B)$ and $j_{n} \in \N$ sufficiently large such that $| v_{n} - u_0|_{q+1, \alpha} \leq \frac{1}{n}$, $\Phi_{{\rm rad}}(u_0) = \Phi_{j_{n}}(u_0)$, $ \Phi_{{\rm rad}}(v_{n})  < \Phi_{{\rm rad}}(u_0)$ and $\Phi_{j_{n}}(v_{n})< \Phi_{j_{n}}(u_0)$;  recall that $u_0 \in L^{\infty}(B)$. In addition, we take $(j_n)$ increasing such that $j_n \rt \infty$ as $n \rt \infty$. Then
\[
\Phi_{j_{n}}(u) \geq \frac{1}{2} \int_B |\nabla u|^2dx  - C \int_B |x|^{\alpha} |u|^{q+1}dx - C, \quad \forall \, u \in H^1_{0, {\rm rad}}(B), \quad \forall \, n \in \N,
\]
and, from the subcritical growth of $F_{j_{n}}$, there exists $u_{n} \in H^1_{0, {\rm rad}}(B)$ such that $| u_{n} - u_0|_{q+1, \alpha} \leq \frac{1}{n}$ and
\[
\Phi_{j_{n}}(u_{n}) = \min_{u \in H^1_{0, {\rm rad}}(B), \, |  u - u_0 |_{q+1, \alpha} \leq \frac{1}{n}}\Phi_{j_{n}}(u).
\]
In particular,
\begin{equation}\label{desigualdades epsilon}
 \Phi_{j_{n}}(u_{n}) \leq  \Phi_{j_{n}}(v_{n})  < \Phi_{{\rm rad}}(u_0) = \Phi_{j_{n}} (u_0),
\end{equation}
and there exists $\mu_{n} \in \R$ such that
\begin{equation}\label{solucao uepsilon}
-\Delta u_{n} = f_{j_{n}}(x,u_{n}) + \mu_{n}|x|^{\alpha}|u_{n} - u_0|^{q -1}(u_{n} - u_0) \ \text{in} \ B, \ u_{n} = 0 \ \text{on} \ \partial B.
\end{equation}
Observe that $\mu_{n}$ is the Lagrange multiplier associated to the minimization of $\Phi_{j_{n}}$ on
\[
C_{n} = \left\{ u \in H^1_{0, {\rm rad}}(B); |u -u_0|_{q+1, \alpha}^{q+1} \leq n^{-(q+1)} \right\}.
\]
Indeed, observe that $|u_n - u_0|_{q+1, \alpha} = \frac{1}{n}$ or $u_n$ is a local minimum for $\Phi_{j_n}: H^1_{0, {\rm rad}}(B) \rt \R$.

 Now,  since $u_{n} + t (u_0 - u_{n}) \in C_ {n}$ for every $t \in [0,1]$, we obtain
\[
0 \leq \Phi_{j_{n}}'(u_{n})(u_0 - u_{n}) = - \mu_{n} \int_B |x|^{\alpha} |u_{n} - u_ 0|^{q+1} dx
\]
and, therefore, $\mu_{n} \leq 0$. In addition, arguing as in Step 1, for every $n \in \N$, it follows that $u_n$ strongly solves \eqref{solucao uepsilon} and $u_{n} \in C^{1, \gamma}(\overline{B})$ for every $0 < \gamma < 1$.

We prove below that
\begin{equation}\label{boa convergencia}
u_{n} \rt u_{0} \quad \text{in} \quad C^1_{0,{\rm rad}}(\overline{B}) \quad \text{as}  \quad n \rt \infty.
\end{equation}

For this moment assume \eqref{boa convergencia}, which in particular means that $(u_{n})$ is bounded in $L^{\infty}(B)$. Then, from \eqref{desigualdades epsilon}, for every $n$ sufficiently large
\[
\Phi(u_n) = \Phi_{j_n}(u_n) < \Phi(u_0)
\]
which contradicts \eqref{minimum C1 introduction}.

Now we prove \eqref{boa convergencia}. We distinguish two cases according to the behavior of $(\mu_{n})$ as $n \rt \infty$, namely,
\begin{enumerate}[(i)]
\item $(\mu_n)$ is bounded;
\item there exists a subsequence of $(\mu_n)$ that converges to $-\infty$.
\end{enumerate}

\nin Case (i). In this case, from \eqref{solucao uepsilon}, $u_n$ strongly solves
\begin{equation}\label{nova eq un}
-\Delta u_n = g_n(x, u_n) \quad \text{in} \quad B, \quad u_n =0 \quad \partial B,
\end{equation}
and also in the sense of $H^1_0(B)$. Here $g_n(x, u_n) = f_{j_{n}}(x,u_{n}) + \mu_{n}|x|^{\alpha}|u_{n} - u_0|^{q-1}(u_{n} - u_0)$. In particular, $(u_n)$ is bounded in $H^1_0(B)$.

Again we split the proof into two cases: $N=1,2$ and $N\geq 3$.

In case $N=1,2$, $H^1_0(B) \hookrightarrow L^t(B)$ continuously for every $t \geq 1$. Then, since $(u_n)$ is bounded in $H^1_0(B)$, for each $t \geq 1$, there exists $\widetilde{C}_t> 0$ such that $|g_n(x,u_n)|_{t} \leq C_t$. Therefore, by the classical elliptic regularity as in \cite[Lemma 9.17]{gilbarg-trudinger}, for each $t\geq 1$ there exists $C_t> 0$ such that $\| u_n \|_{W^{2,t}} \leq C_t$ for all $n \in \N$. Therefore, \eqref{boa convergencia} follows from the classical Sobolev embeddings.

In case $N \geq 3$, since $H^1_{0, {\rm rad}}(B) \hookrightarrow L^{2^*_{\alpha}}(B, |x|^{\alpha})$ continuously and $(u_n)$ is bounded in $H^1_{0, {\rm rad}}(B)$, there exists $a(x) \in L^{\frac{N}{2}}(B)$, see \eqref{legal des regularidade}, such that
\begin{multline*}
| g_n(x,u_n(x))| = |f_{j_{n}}(x,u_{n}) + \mu_{n}|x|^{\alpha}|u_{n} - u_0|^{2^*_{\alpha} -2}(u_{n} - u_0)| \\\leq a(x)(1 + |u_n(x)|), \quad \forall \, n \in \N, \quad \forall \, x \in B.
\end{multline*}
Then, arguing as in \cite[Proposition 1.2]{gueda-veron} and by \cite[Lemma 9.17]{gilbarg-trudinger}, for each $t \geq 1$ there exist $C_t > 0$ such that $\| u_n \|_{W^{2,t}} \leq C_t$ for all $n \in \N$. Therefore, once more, \eqref{boa convergencia} follows from the classical Sobolev embeddings.

\medbreak
\nin Case (ii). To simplify notation, we also denote by $(\mu_n)$ the subsequence of $(\mu_n)$ that converges to $-\infty$. Observe that \eqref{crescimento f regularidade introduction} and the fact that $u_0$ is bounded, guarantees the existence $n_0 \in \N$ and $M>0$, that do not depend on $n$, such that
\[
g_n(x, s) = f_{j_{n}}(x,s) + \mu_{n}|x|^{\alpha}|s- u_0|^{q-1}(s - u_0) \left\{ \begin{array}{l} < 0 \quad \text{if} \quad s \geq M, \\ > 0 \quad \text{if} \quad s \leq -M,\end{array}\right.,
\]
forall $x \in B\menos\{0\}$ and all $n \geq n_0$. Set $u^+ = \max\{ u,0\}$ and $u^- = \max\{ -u,0\}$. Then $(u_n - M)^+, (u_n+M)^- \in H^1_0(B)$ and taking $(u_n - M)^+, (u_n+M)^-$ as test functions in \eqref{nova eq un} we obtain that $|u_n(x)| \leq M$ for every $n \geq n_0$ and for every $x \in \overline{B}$.

Let $s> 1$. Then, from \eqref{general equation introduction} and \eqref{nova eq un}, with $ |u_{n} - u_0|^{s-1}(u_n -u_0)$ as a test function we obtain
\begin{multline*}
0 \leq s \int_B |u_n - u_0|^{s-1}| \nabla u_n - u_0|^2 dx \\= \int_B \left( f_{j_n}(x, u_n) - f(x, u_0) \right) |u_n - u_0|^{s-1}(u_n - u_0) dx + \mu_n \int_B |x|^{\alpha}|u_n - u_0|^{q + s} dx.
\end{multline*}
Now, since $(u_n)$ remains bounded in $L^{\infty}$, from Hölder's inequality with $ p  = \frac{q + s}{s}$ applied to the above inequality we obtain
\begin{multline*}
- \mu_n \int_B |u_n - u_0|^{q + s} |x|^{\alpha} dx  \leq C \int_B |u_n - u_0|^s |x|^\alpha dx \\\leq C \left( \int_B |x|^{\alpha}dx \right)^{\frac{p-1}{p}} \left(\int_B |u_n - u_0|^{q + s}|x|^{\alpha} dx\right)^{\frac{1}{p}}
\end{multline*}
and therefore
\[
(- \mu_n)^{\frac{1}{q}} \left(  \int_B |u_n - u_0|^{q + s}|x|^{\alpha}dx \right)^{\frac{1}{q+ s }} \leq C \left( \int_B |x|^{\alpha}dx \right)^{\frac{1}{q+ s}}.
\]
Finally, taking the limit as $s \rt \infty$ in the above inequality we obtain that $(-\mu_n||u_n - u_0|^{q}|_{\infty})$ is bounded and hence we argue as in Case (i).
\end{proof}

\begin{remark}
[\textbf{$H^1$ versus $C^1$ local minimizers: partially symmetric functions}] The proof Theorem {\rm \ref{HXC partial introduction}} follows similarly as the proof of Theorem {\rm \ref{teo HXC radial introduction}}, basically replacing Theorem {\rm \ref{teo regularidade radial}} by Theorem {\rm \ref{teo regularidade l}}, hence omitted here.
\end{remark}

\section{Existence results and local minimum solutions}\label{section existencia}

In this part we start our study on the boundary value problem \eqref{problem weight introduction}. Set
%\begin{equation}\label{primeiro autovalor}
\[
\gd{ \lambda_{1, \alpha} = \inf_{u \in H^1_0(B)\menos\{0\}} \frac{\int_B |\nabla u|^2 dx}{\int_B |x|^{\alpha}| u|^2 dx}. }
\]
%\end{equation}
Then $\lambda_{1,\alpha}$ is the first eigenvalue for the weighted eigenvalue problem
%\begin{equation}\label{autovalor}
\[
-\Delta u = \lambda |x|^{\alpha}u \quad \text{in} \quad B, \quad \text{with} \quad u = 0 \quad \text{on} \quad \partial B,
\]
%\end{equation}
which is simple and has a positive eigenfunction $\varphi_{1,\alpha}$. Consequently, $\varphi_{1,\alpha}$ is radially symmetric and strictly radially decreasing. Here we consider $\varphi_{1,\alpha}$ with $\int |x|^{\alpha} \varphi_{1,\alpha}^2 dx =1$.

%If $u = v+a$, then \eqref{problema} is equivalent to
%\begin{equation}\label{equiv1}
% -\Delta v = |x|^{\alpha}|v +a|^{p-1}(v+a) \quad \text{in} \quad B, \quad v > 0 \quad \text{in} \quad B, \quad v = 0 \quad \text{on} \quad \partial B,
%\end{equation}
%and with $v = aw$, then \eqref{equiv1} is equivalent to

%\begin{equation}\label{problem weight}
% -\Delta w = \lambda |x|^{\alpha} | w + 1|^{p-1}(w +1) \quad \text{in} \quad B, \quad w> 0 \quad \text{in} \quad B, \quad w =0 \quad \text{on} \quad \partial B.
%\end{equation}

%We will also work with both \eqref{equiv1} and \eqref{problem weight} in order to obtain our results on \eqref{problema}.

\begin{lemma}
If  $\lambda \geq \frac{\lambda_{1,\alpha}}{p}$, then \eqref{problem weight introduction} has no solution.
 \end{lemma}
\begin{proof}
If $w$ is a solution of \eqref{problem weight introduction}, then
 \[
 \lambda_{1, \alpha} \int_B |x|^{\alpha} \varphi_1 w \, dx = \int_B \nabla \varphi_1 \nabla w \, dx = \lambda \int_B | x |^{\alpha} (w +1)^p\varphi_1 \, dx > \lambda p \int_B |x|^\alpha \varphi_1 w \, dx,
 \]
 which implies $\lambda p < \lambda_{1,\alpha}$.
\end{proof}

Let $e_{\alpha}$ and $e_0$ be the solutions of
\[
-\Delta e_{\alpha} = |x|^{\alpha}, \ \ -\Delta e_0 =1 \quad \text{in} \quad B, \quad e_{\alpha} = e_{0} = 0 \quad \text{on} \quad \partial B.
\]
Then, by the strong maximum principle, $0 < e_{\alpha} < e_0$ in $B$ for all $\alpha > 0$.
\begin{lemma}\label{lemmasubsuper}
 If $0 < \lambda \leq \frac{(p-1)^{p-1}}{p^p | e_{\alpha}|_{\infty}}$, then \eqref{problem weight introduction} has a solution.
\end{lemma}
\begin{proof}
First observe that, for every $\lambda>0$, $\underline{w} = \lambda e_{\alpha}$ is a lower solution of \eqref{problem weight introduction}. Now we search for an upper solution $\overline{w}$ of the form $\overline{w} = M e_{\alpha}$. Then $\overline{w} = M e_{\alpha}$ is an upper solution for \eqref{problem weight introduction} if, and only if,
 \begin{equation}\label{estimativa1}
 \frac{(1 + M |e_{\alpha}|_{\infty})^p}{M} \leq \frac{1}{\lambda}.
 \end{equation}
 Let
 \[
 h(t) = \frac{(1+ t | e_{\alpha}|_{\infty})^p}{t}, \quad t > 0.
 \]
 Observe that $h$ attains its minimum value $\frac{p^p | e_{\alpha}|_{\infty}}{(p-1)^{p-1}}$ at $t_* = \frac{1}{(p-1)| e_{\alpha}|_{\infty}}$. Therefore, if $0 < \lambda \leq \frac{(p-1)^{p-1}}{p^p | e_{\alpha}|_{\infty}}$, then $\overline{w} = M e_{\alpha}$ with $M = t_*$ is an upper solution for \eqref{problem weight introduction} such that $\underline{w} \leq \overline{w}$.  Then we can apply the monotone iteration technique, as in \cite{amann}, to prove the existence of a solution to \eqref{problem weight introduction}.
\end{proof}

\begin{lemma}\label{upper zero}
There exists $\lambda_0 = \lambda_0(N,p) \in (0,1)$ such that for every $\lambda \in (0, \lambda_0)$, every integer $k \geq 2$, then $\overline{w} = \lambda^{1/k} e_{\alpha}$ is an upper solution for \eqref{problem weight introduction} such that $\underline{w} = \lambda e_{\alpha} < \lambda^{1/k} e_{\alpha} = \overline{w}$ in $B$.
\end{lemma}
\begin{proof}
 It is a straightforward consequence of \eqref{estimativa1}. Observe that $| e_{\alpha}|_{\infty} \leq | e_{0}|_{\infty}$ for every $\alpha\geq 0$ and this allows us to get $ \lambda_0(N,p)$ independent of $\alpha$.
\end{proof}

We say that a solution $w_0$ of
\begin{equation}\label{problemagrealauxiliar}
-\Delta w = f(x,w) \quad \text{in} \quad \Omega, \quad w > 0 \quad \text{in} \quad \Omega, \quad w = 0 \quad \text{on} \quad \partial \Omega,
\end{equation}
is a minimal solution if
\[
w_{0}(x) \leq w(x) \quad \forall \ x \in B
\]
for any other solution $w$ of \eqref{problemagrealauxiliar}.

Let
\[
\lambda_* = \sup \{ \lambda> 0; \eqref{problem weight introduction} \,\, \text{has a solution}  \}.
\]

To treat \eqref{problem weight introduction} we will also consider an equivalent problem, namely, if we write $v = aw$, then \eqref{problem weight introduction} is equivalent to
\begin{equation}\label{equiv1}
 -\Delta v = |x|^{\alpha}(v +a)^{p} \ \text{in} \ B, \ v > 0 \ \text{in} \ B, \ v = 0 \ \text{on} \ \partial B, \ \text{with} \ a^{p-1} = \lambda.
\end{equation}

\begin{proposition}\label{existence open interval}Let $\lambda_*$ be as above. Then $0 < \lambda_* < \infty$ and:
\begin{enumerate}[(i)]
 \item There is no solution of \eqref{problem weight introduction} if $\lambda> \lambda_*$;
 \item If $0 < \lambda < \lambda_*$, then \eqref{problem weight introduction} has a minimal solution $w_{\alpha,\lambda}$. Let $v_{\alpha, \lambda}$ and $w_{\alpha, \lambda}$ be the corresponding minimal solutions of \eqref{equiv1} and \eqref{problem weight introduction}, respectvely. Then $w_{\alpha, \lambda} \rt 0$ uniformly w.r.t. $x$ and $\alpha$ as $\lambda \rt 0$.
 \item $v_{\alpha, \lambda}, w_{\alpha, \lambda}$ are positive, radially symmetric and radially decreasing.
\end{enumerate}
\end{proposition}
\begin{proof}
Regarding the existence on minimal solution, we apply the monotone iteration technique, as in \cite{amann}. Indeed the minimal solution $w_{\alpha, \lambda}$ is obtained by the lower and upper solution method with a monotone iteration starting from the lower solution $\underline{w} = \lambda e_{\alpha}$ and  therefore it is radially symmetric and radially decreasing.

Fix an integer $k \geq 2$. Then, by Lemma {\rm \ref{upper zero}} and the strong maximum principle, we conclude that
\begin{equation}\label{asymptotics}
w_{\alpha, \lambda} (x) < \lambda^{1/k} e_{\alpha} (x) < \lambda^{1/k} \| e_0 \|_{\infty} \quad \forall \, x \in B \ \ \text{and} \ \ \lambda \in (0, \lambda_0(N,p)).
\end{equation}

Items (i) and (iii) are quite evident. 
\end{proof}

\begin{proposition}\label{sol minimais} Let $v_{\alpha, \lambda}$ and $w_{\alpha, \lambda}$ be the corresponding minimal solutions of \eqref{equiv1} and \eqref{problem weight introduction}, respectvely. Then
 \[
 \lim_{\lambda \rt 0 } \frac{w_{\alpha, \lambda}}{\lambda e_{\alpha}} = 1, \quad  \lim_{\lambda \rt 0 } \frac{v_{\alpha, \lambda}}{\lambda^{p/(p-1)} e_{\alpha}} = 1, \quad \text{uniformly w.r.t.} \:\: x \:\: \text{and} \:\: \alpha.
 \]
\end{proposition}
\begin{proof}
By the strong maximum principle we have  that $\lambda e_{\alpha} < w_{\alpha, \lambda}$ in $B$. On the other hand, given $\epsilon >0$, we have by \eqref{asymptotics} that there exists $\lambda_{\epsilon}> 0$ such that
 \[
 |w_{\alpha, \lambda}|_{\infty} < \epsilon, \quad \forall \; 0 < \lambda < \lambda_{\epsilon}, \,\, \forall \, \alpha >0.
 \]
 Then
 \[
 - \Delta w_{\alpha, \lambda} = \lambda|x|^{\alpha}(1 + w_{\alpha, \lambda})^p < \lambda | x|^{\alpha} (1 + \epsilon)^p = - (1 + \epsilon)^p \Delta (\lambda e_{\alpha}),
 \]
and, by strong maximum principle, we obtain
 \[
 \lambda e_{\alpha} < w_{\alpha, \lambda} < (1+ \epsilon)^p \lambda e_{\alpha}, \quad \forall \, 0 < \lambda < \lambda_{\epsilon},  \, \forall \, \alpha>0.
 \]

For the second limit observe that $v_{\alpha, \lambda} = a w_{\alpha, \lambda}$ with $a^{p-1} = \lambda$. 
\end{proof}

\subsection{On radial local minimum solutions and the proof Theorem \ref{teorema existencia introduction}}
Let $1 < p$ and $0< \lambda < \lambda_*$. In case $N\geq 3$ also assume $p \leq 2^*_{\alpha} -1$. Then, see Remark \ref{classical regularity radial}, we know that the (positive) critical points of the $C^1(H^1_{0, {\rm rad}}(B), \R)$-functional
\begin{equation}\label{funcional radial}
J_{\lambda,{\rm rad}}(v) = \frac{1}{2}\int_B |\nabla v |^2 dx - \frac{1}{p+1} \int_B |x|^{\alpha} | v + a|^{p+1} dx,  \ a^{p-1} = \lambda,
\end{equation}
are precisely the classical radial solutions of \eqref{equiv1}.
%Now we go back to the functional $J_{{\rm rad}}$.

\begin{proposition}\label{prop energia negativa}
Let $1< p$ and in case $N\geq 3$ also assume $p \leq 2^*_{\alpha} -1$. If $0< \lambda < \lambda_*$, then there exists $\widetilde{v}_{\alpha,{\rm rad}, \lambda} >0$ in $B$ such that $J_{\lambda, {\rm rad}}$ has a local minimum at $\widetilde{v}_{\alpha,{\rm rad}, \lambda}$, $J_{\lambda, {\rm rad}}(\widetilde{v}_{\alpha,{\rm rad}, \lambda})< 0$ and
\begin{equation}\label{barra comportamento}
 \lim_{\lambda \rt 0 } \frac{\widetilde{v}_{\alpha,{\rm rad}, \lambda}}{\lambda^{p/(p-1)} e_{\alpha}} = 1, \quad \text{uniformly w.r.t.} \:\: x \:\: \text{and} \:\: \alpha.
 \end{equation}
\end{proposition}
\begin{proof}
Given $0 < \lambda < \lambda_*$ fix any $\lambda' \in (\lambda, \lambda_*)$. Then observe that $0$ and $v_{\alpha,\lambda'}$ are, respectively, strict lower and upper solutions of \eqref{equiv1}.  Consider
\[
[0, v_{\alpha,\lambda'}] = \{ v \in C^1_{0, {\rm rad}}(B); 0 \leq v \leq v_{\alpha,\lambda'} \:\: \text{in}\:\:B  \}.
\]
Then arguing as \cite[Theorem 2.4]{struwe}, there exists $\widetilde{v}_{\alpha,{\rm rad}, \lambda} \in [0, v_{\alpha, \lambda'}]$ such that
\[
J_{\lambda, {\rm rad}}(\widetilde{v}_{\alpha,{\rm rad}, \lambda}) = \min_{v \in [0, v_{\alpha, \lambda'}]}J_{\lambda, {\rm rad}}(v).
\]
Then, by Theorems \ref{teo HXC radial introduction} and \ref{teo regularidade radial},  $\widetilde{v}_{\alpha,{\rm rad}, \lambda}$ is a classical solution of \eqref{equiv1} and, by the strong maximum principle, $0 < \widetilde{v}_{\alpha,{\rm rad}, \lambda} < v_{\alpha,\lambda'}$ in $B$.
%Consequently,
%$J_{{\rm rad}}$ has a local minimum at $\widetilde{v}_{\alpha,a}$ in the $C^1_{{\rm rad}}$ topology. Then, by Proposition \ref{HXC},
Then  $\widetilde{v}_{\alpha,{\rm rad}, \lambda}$ is a local minimum of $J_{\lambda, {\rm rad}}$ in the $H^1_{0, {\rm rad}}(B)$ topology and, since $J_{\lambda, {\rm rad}}(0) < 0$, it follows that $J_{\lambda, {\rm rad}}(\widetilde{v}_{\alpha,{\rm rad}, \lambda})< 0$. Finally, from $v_{\alpha, \lambda} \leq \widetilde{v}_{\alpha,{\rm rad}, \lambda} < v_{\alpha, \lambda'}$ in $B$, we obtain \eqref{barra comportamento} from Proposition \ref{sol minimais}. 
\end{proof}

\begin{remark}
We do not know whether or not $\widetilde{v}_{\alpha,{\rm rad}, \lambda} = v_{\alpha, \lambda}$. See \cite{keener-keller} and also \cite[eq. (2.76)]{brezis-nirenberg} for a similar problem where it is proved that the minimal positive solution is the local minimum solution.
\end{remark}

\begin{proposition}\label{a priori estimates radial}
 Let $N \geq 1$, $\alpha >0$,  $\lambda \in (0, \lambda_*]$, $1 < p$ and in case $N \geq 3$ also assume $p \leq  2^*_{\alpha}-1$. Fix $0 < \gamma < \min\{ \alpha, 1\}$. Then there exists a positive constant $C$ such that any positive radial solution $v$ of {\rm \eqref{equiv1}} with $J_{\lambda,{\rm rad}}(v)<0$ satisfies
 \[
 \|v \|_{0, \gamma} < C.
 \]
As a consequence, \eqref{problem weight introduction} has a solution with $\lambda = \lambda_*$.
\end{proposition}
\begin{proof}
We recall that $a^{p-1} = \lambda$. We have that
\begin{multline*}
 0> J_{\lambda, {\rm rad}}(v) = \frac{1}{2}\int_B | \nabla v|^2 dx - \frac{1}{p+1}\int_B |x|^{\alpha} (v + a)^{p+1} dx \\
 = \frac{1}{2(p+1)}\int_B |x|^{\alpha} \left( (p-1)(v + a)^{p+1} - (v + a)^p 2a) \right)dx,
\end{multline*}
 and then,
 \[
 \int_B|x|^{\alpha}|v + a|^{p+1} dx \leq C,
 \]
with $C$ independent of $\alpha$. 
Once more using that $J_{{\rm rad}}(v)<0$, we get  
\begin{equation}\label{independente1}
\| v\| < C \quad \forall \ 0 < \lambda < \lambda_*,
\end{equation}
again with $C$ independent of $\alpha$. Then, arguing as in the proof of Theorem \ref{teo HXC radial introduction} based on \cite[Proposition 1.2]{gueda-veron}, we get $\| v\|_{\infty} \leq C$ for all $0< \lambda < \lambda^*$. 

Now, from Proposition \ref{prop energia negativa}, we have that $J_{\lambda, {\rm rad}}(\widetilde{v}_{\alpha,{\rm rad}, \lambda})< 0$.  Then using that $0< v_{\alpha, \lambda} \leq \widetilde{v}_{\alpha,{\rm rad}, \lambda}$ for all $0< \lambda< \lambda_*$ and the above a priori bounds, we obtain the existence of a solution to \eqref{equiv1} with $\lambda = \lambda_*$ as a limit of $v_{\alpha, \lambda}$ as $\lambda \uparrow \lambda_*$. 
 \end{proof}

\begin{proof}[\textbf{Proof of Theorem \ref{teorema existencia introduction}}] 
It follows directly from Proposition \ref{existence open interval} and \ref{a priori estimates radial}.
\end{proof}

%xxxxxxxxxxxxxxxxxxxxxxxxx
\subsection{On partially symmetric local minimum solutions}
Let $1 < p$, $0< \lambda < \lambda_*$ and assume all the hypotheses of Theorem \ref{teo regularidade l}. Then from Remark \ref{remark regularity l}, we know that the (positive) critical points of the $C^1(H_l(B), \R)$-functional
\begin{equation}\label{funcional l}
J_{\lambda, l}(v) = \frac{1}{2}\int_B |\nabla v |^2 dx - \frac{1}{p+1} \int_B |x|^{\alpha} | v + a|^{p+1} dx, \quad v \in H_l(B), \quad a^{p-1} = \lambda,
\end{equation}
are precisely the classical $(l, N-l)$-symmetric solutions of \eqref{equiv1}.

\begin{proposition}\label{prop energia negativa l}
Let $1< p$ and assume all the hypotheses of Theorem {\rm \ref{teo regularidade l}}. If $0< \lambda < \lambda_*$, then exists $\widetilde{v}_{\alpha, l, \lambda} >0$ in $B$ such that $J_{\lambda, l}$ has a local minimum at $\widetilde{v}_{\alpha, l, \lambda}$, $J_{\lambda, l}(\widetilde{v}_{\alpha,l, \lambda})< 0$ and
\begin{equation}\label{barra comportamento l}
 \lim_{\lambda \rt 0 } \frac{\widetilde{v}_{\alpha,l, \lambda}}{\lambda^{p/(p-1)} e_{\alpha}} = 1, \quad \text{uniformly w.r.t.} \:\: x \:\: \text{and} \:\: \alpha.
 \end{equation}
\end{proposition}
\begin{proof}
It is very similar to the proof of Proposition \ref{prop energia negativa}, at this time with the aid of Theorem \ref{HXC partial introduction}. 
\end{proof}

\subsection{On local minimum solutions without symmetry}
Let $1 < p$, $N \geq 1$, $0< \lambda < \lambda_*$ and assume $1< p<2^*-1$. Then we know that classical solutions of \eqref{equiv1} are precisely the (positive) critical points of the $C^1(H^1_0(B), \R)$-functional
\begin{equation}\label{funcional h10}
J_{\lambda}(v) = \frac{1}{2}\int_B |\nabla v |^2 dx - \frac{1}{p+1} \int_B |x|^{\alpha} | v + a|^{p+1} dx,  \quad a^{p-1} = \lambda.
\end{equation}

\begin{proposition}\label{prop energia negativa h10}
Let $1 < p$, $N \geq 1$ and assume $1< p<2^*-1$. If $0< \lambda < \lambda_*$, then exists $\widetilde{v}_{\alpha, \lambda} >0$ in $B$ such that $J_{\lambda}$ has a local minimum at $\widetilde{v}_{\alpha, \lambda}$, $J_{\lambda}(\widetilde{v}_{\alpha, \lambda})< 0$ and
\begin{equation}\label{barra comportamento h10}
 \lim_{\lambda \rt 0 } \frac{\widetilde{v}_{\alpha, \lambda}}{\lambda^{p/(p-1)} e_{\alpha}} = 1, \quad \text{uniformly w.r.t.} \:\: x \:\: \text{and} \:\: \alpha.
 \end{equation}
\end{proposition}
\begin{proof}
It is very similar to the proof of Proposition \ref{prop energia negativa}, at this time with the aid of \cite{brezis-nirenbergHXC}. 
\end{proof}

\begin{remark}
We do not know whether or not $\widetilde{v}_{\alpha,{\rm rad}, \lambda} = v_{\alpha, \lambda} = \widetilde{v}_{\alpha, l, \lambda} = \widetilde{v}_{\alpha, \lambda}$.
\end{remark}

%xxxxxxxxxxxxxxxxxxxxxxxxxx

\section{Multiple positive solutions: proof of Theorem \ref{teorema multiplicidade introduction}} \label{section multiplicidade}

\subsection{Proof of Theorem {\rm \ref{teorema multiplicidade introduction} (I)}:  under the extra assumption $1 < p < 2^*_{\alpha}-1$ in case $N \geq 3$}\label{subcritical section}

In this case $J_{\lambda, {\rm rad}}$, as defined by \eqref{funcional radial}, satisfies the (PS) condition. In addition, $J_{\lambda, {\rm rad}}$ has a local minimum at $\widetilde{v}_{\alpha,{\rm rad}, \lambda}$, with $\widetilde{v}_{\alpha,{\rm rad}, \lambda}$ given by Proposition \ref{prop energia negativa}. We recall that $-\Delta \widetilde{v}_{\alpha,{\rm rad}, \lambda} >0$ in $B$, $\widetilde{v}_{\alpha,{\rm rad}, \lambda} > 0$ in $B$ and that $J_{\lambda, {\rm rad}}(\widetilde{v}_{\alpha,{\rm rad}, \lambda}) <  0$. Let $r_{\lambda}>0$ such that
\begin{equation}\label{minmin}
J_{\lambda, {\rm rad}}(\widetilde{v}_{\alpha,{\rm rad}, \lambda}) \leq J_{\lambda, {\rm rad}}(v) \quad \forall \, v \in H^1_{0, \rm{rad}}(B) \ \ \text{s.t.} \ \ \| v - \widetilde{v}_{\alpha,{\rm rad}, \lambda}\| < r_{\lambda}.
\end{equation}

We have to consider two cases.
\medbreak 
\noindent\textbf{Case 1:} There exist $\epsilon_\lambda >0$ and $r \in (0, r_{\lambda})$ such that
\[
J_{\lambda, {\rm rad}}(v) >J_{\lambda, {\rm rad}}(\widetilde{v}_{\alpha,{\rm rad}, \lambda}) + \epsilon_\lambda, \quad \forall \, v\in H^1_{0, {\rm rad}}(B), \, \| v - \widetilde{v}_{\alpha,{\rm rad}, \lambda} \|  = r.
\]
Choose $v_0 \in H^1_{0, {\rm rad}}(B)$ such that $v_0 \geq 0$ in $B$, $\|v_0 - \widetilde{v}_{\alpha,{\rm rad}, \lambda}\| > r_\lambda$ and $J_{\lambda, {\rm rad}}(v_0) <  J_{\lambda, {\rm rad}}(\widetilde{v}_{\alpha,{\rm rad}, \lambda})$.  Set
\[
\Gamma_{\alpha, \lambda, {\rm rad}} = \{ \gamma \in C([0,1], H^1_{0, {\rm rad}}(B)); \gamma(0) = v_{\alpha, \lambda} \,\, \text{and} \,\, \gamma(1) = v_0\} \quad \text{and}
\]
\begin{equation}\label{mountain pass level radial}
m_{\alpha, \lambda, {\rm rad}} = \inf_{\gamma \in \Gamma_{\alpha, \lambda, {\rm rad}}} \max_{t \in [0,1]} J_{\lambda, {\rm rad}}(\gamma(t)).
\end{equation}
Now observe that $J_{\lambda, {\rm rad}}(|v|) \leq J_{\lambda, {\rm rad}}(v)$ for every $v \in H^1_{0, {\rm rad}}(B)$. Then, from \cite[Theorem 2.8]{willem}, any mountain pass solution of $J_{\lambda, {\rm rad}}$ associated to the mountain pass level $m_{\alpha, \lambda, {\rm rad}}$ is positive in $B$. Take $V_{\alpha,\lambda, {\rm rad}}$ a mountain pass critical point of $J_{\lambda, {\rm rad}}$ associated to the mountain pass level $m_{\alpha,\lambda, {\rm rad}}$. Then, $V_{\alpha, \lambda, {\rm rad}}$ is a classical solution of \eqref{equiv1} such that $J_{\lambda, {\rm rad}}(V_{\alpha, \lambda, {\rm rad}}) > J_{\lambda, {\rm rad}}(\widetilde{v}_{\alpha,{\rm rad}, \lambda})$. Consequently, \eqref{problem weight introduction} has at least two radial solutions.

\medbreak
\noindent \textbf{Case 2:} Suppose that Case 1 does not hold.
${}$
\noindent In this case, for every $r \in (0, r_{\lambda})$,
\[
\inf\{ J_{\lambda, {\rm rad}}(v); \| v - \widetilde{v}_{\alpha,{\rm rad}, \lambda}\| = r \} = J_{\lambda, {\rm rad}}(\widetilde{v}_{\alpha,{\rm rad}, \lambda}).
\]
Therefore, by \cite[Theorem 5.10]{djairo-LecturesEkeland}, for every $r \in (0, r_{\lambda})$ there exists $v_{0,r} \in H^1_{0, \rm{rad}}(B)$ such that
\[
\| v_{0,r} - \widetilde{v}_{\alpha,{\rm rad}, \lambda}\| = r \quad \text{and} \quad J_{\lambda, {\rm rad}}(\widetilde{v}_{\alpha,{\rm rad}, \lambda}) = J_{\lambda, {\rm rad}}(v_{0,r}).
\]
Then we have that each $v_{0,r}$ is also a local minimum of $J_{\lambda, {\rm rad}}$ and therefore a classical solution of 
\begin{equation}\label{auxaux}
 -\Delta v = |x|^{\alpha}|v +a|^{p-1}(v+a) \ \text{in} \ B, \ \ v = 0 \ \text{on} \ \partial B, \ \text{with} \ a^{p-1} = \lambda.
\end{equation}
We claim that each $v_{0,r}\geq0$ in $B$ and therefore (by the strong maximum principle) another radial solution of \eqref{equiv1} as desired. By contradiction suppose that ``$v_0\geq0$ in $B$'' is not satisfied. Then observe that
\begin{multline*}
\int_B |\nabla |v_0| - \nabla \widetilde{v}_{\alpha,{\rm rad}, \lambda}|^2 dx = \int_B |\nabla v_0 - \nabla \widetilde{v}_{\alpha,{\rm rad}, \lambda}|^2 dx - 2 \int_B \nabla (|v_0| - v_0) \nabla \widetilde{v}_{\alpha,{\rm rad}, \lambda} dx\\
= \int_B |\nabla v_0 - \nabla \widetilde{v}_{\alpha,{\rm rad}, \lambda}|^2 dx - 2 \int_B (|v_0| - v_0) (- \Delta \widetilde{v}_{\alpha,{\rm rad}, \lambda}) dx \leq \int_B |\nabla v_0 - \nabla \widetilde{v}_{\alpha,{\rm rad}, \lambda}|^2 dx
\end{multline*}
that is
\[
\| |v_0| - \widetilde{v}_{\alpha,{\rm rad}, \lambda}\| \leq \| v_0 - \widetilde{v}_{\alpha,{\rm rad}, \lambda}\| < r
\]
and 
\[
J_{\lambda, {\rm rad}}(|v_0|) < J_{\lambda, {\rm rad}}(v_0) = J_{\lambda, {\rm rad}}(\widetilde{v}_{\alpha,{\rm rad}, \lambda}),
\]
which contradicts \eqref{minmin}. Therefore $v_0\geq 0$ in $B$ and, by the strong maximum principle, we infer that $v_0>0$ in $B$. Consequently, \eqref{problem weight introduction} has at least two radial solutions.

\subsection{Proof of Theorem {\rm \ref{teorema multiplicidade introduction} (I)}\label{duassol}:  with $N \geq 3$ and $p = 2^*_{\alpha}-1$}

In this case $J_{\lambda, {\rm rad}}$ does not satisfies the (PS) condition at all levels. To overcome this difficulty we adopt an approach different from that in Section \ref{subcritical section} and close to the proof of \cite[Theorem 2.5]{djairo-gossez-ubilla}.

For that, it is essential to study
\begin{equation}\label{todo rn}
-\Delta u = |x|^{\alpha} u^{2^*_{\alpha}-1} \quad \text{in} \quad \R^N, \quad u>0 \quad \text{in} \quad \R^N,
\end{equation}
and in this direction we cite the pioneering work \cite{fowler}. Let
\[
\mathcal{D}^{1,2}_{{\rm rad}}(\R^N) = \left\{  u \in \mathcal{D}^{1,2}(\R^N), u \: \text{is radially symmetric}\right\}.
\]
Arguing as in \cite{ni}, \cite{ederson-djairo-olimpio}, see also \cite{strauss}, we know that every element $u \in \mathcal{D}^{1,2}_{{\rm rad}}(\R^N)$ has a representative $U$, also denoted by $u$, which is continuous in $\R^N \menos\{0\}$ and satisfies
\[
\gd{ |u(x)| \leq \frac{1}{[(N-2)\omega_{N}]^{1/2}} \frac{\left(\int_{\R^N}|\nabla u| dx\right)^{1/2}}{|x|^{\frac{N-2}{2}}}, \quad \forall \, x \in \R^N\menos\{0\}. }
\]
At this time we use the identity $u(x) = - \int_{|x|}^{\infty} u'(t) dt$.
Then using the embedding $\mathcal{D}^{1,2}(\R^N) \hookrightarrow L^{2^*}(\R^N)$ and arguing as in \cite[p. 3742]{ederson-djairo-olimpio} we conclude that $\mathcal{D}^{1,2}_{{\rm rad}}(\R^N) \hookrightarrow L^{2^*_{\alpha}}(\R^N, |x|^{\alpha})$.

Now consider
\begin{equation}\label{minimization rn}
\gd{S_{\alpha}(\R^N) = \inf_{u \in {\mathcal D}^{1,2}_{{\rm rad}}(\R^N) \menos \{ 0\} } \frac{\int_{\R^N} |\nabla u|^2 dx}{ \left( \int_{\R^N} |u|^{2^*_{\alpha}} |x|^{\alpha}dx\right)^{2/2^*_{\alpha}}}}.
\end{equation}

We know that any solution of $S_{\alpha}(\R^N)$ induces, up to multiplication by a constant, a $C^2(\R^N)$ radial solution of \eqref{todo rn}.
On the other hand, it is proved in \cite[eq. (A.9)]{gidas-spruck} that all $C^2(\R^N)$ radial solutions of \eqref{todo rn} are explicitly given by
\begin{equation}\label{familia de solucoes}
 \mathfrak{u}_{\alpha,\theta}(x) = \left[\frac{ \sqrt{\theta(N-2)(N+\alpha)}}{\theta + |x|^{\alpha + 2}} \right]^{\frac{N-2}{\alpha + 2}},  \quad x\in \R^N, \,\, \theta>0.
\end{equation}

\begin{lemma}\label{invariancia}
Given $R > 0$ let $B_R = \{ x \in \R^N; |x| < R\}$. Set
\[
\gd{S_{\alpha}(B_R) = \inf_{u \in {H}^{1}_{0,{\rm rad}}(B_R) \menos \{ 0\} } \frac{\int_{B_R} |\nabla u|^2 dx}{ \left( \int_{B_R} |u|^{2^*_{\alpha}} |x|^{\alpha}dx\right)^{2/2^*_{\alpha}}}}.
\]
Then $S_{\alpha}(B_R)=S_{\alpha}(\R^N)$.
\end{lemma}
\begin{proof}
Extending by zero an element in ${H}^{1}_{0,{\rm rad}}(B_R)$ we obtain an element in ${\mathcal D}^{1,2}_{{\rm rad}}(\R^N)$. Then $S^{\alpha}(\R^N) \leq S_{\alpha}(B_R)$.

Now, let $(w_n) \con {\mathcal D}_{{\rm rad}}(\R^N)$ be a minimizing sequence for $S_{\alpha}(\R^N)$ such that, for each $n$, ${\rm supp}(w_n) \con B_{r_n}(0)$. Then $u_n(x) = \left(\frac{r_n}{R}\right)^{\frac{N-2}{2}}w\left( \frac{r_n}{R} x\right) \in H^1_{0, {\rm rad}}(B)$ and
\[
\frac{\int_{B_R} | \nabla u_n|^2 dx}{\left( \int_{B_R} | u_n|^{2^*_{\alpha}}|x|^{\alpha} dx \right)^{2/2^*_{\alpha}}} =\frac{\int_{\R^N} | \nabla w_n|^2 dx}{\left( \int_{\R^N} | w_n|^{2^*_{\alpha}}|x|^{\alpha} dx \right)^{2/2^*_{\alpha}}}.
\]
Therefore, $S_{\alpha}(B_R) = S_{\alpha}(\R^N)$. 
\end{proof}

In view of Lemma \ref{invariancia} we will denote $S_{\alpha}(\R^N)$ and $S_{\alpha}(B_R)$ by $S_{\alpha}$.

%xxxxxxxxxxxxxxxxxxxxxxxxxxxxxxxxxxxxxxxx

Fix $\lambda$ with $0 <\lambda<\lambda_*$. We recall that, with $a^{p-1} = \lambda$,
\[
J_{\lambda, {\rm rad}}(v) = \frac{1}{2}\int_B |\nabla v |^2 dx - \frac{1}{p+1} \int_B |x|^{\alpha} | v + a|^{p+1} dx, \quad v \in H^1_{0, {\rm rad}}(B),
\]
has a local minimum at $\widetilde{v}_{\alpha,{\rm rad}, \lambda}$, $J_{\lambda, {\rm rad}}(\widetilde{v}_{\alpha,{\rm rad}, \lambda}) <  0$, $-\Delta \widetilde{v}_{\alpha,{\rm rad}, \lambda} > 0$ in $B$ and $\widetilde{v}_{\alpha,{\rm rad}, \lambda} > 0$ in $B$; see Proposition \ref{prop energia negativa}. It is enough to prove the existence of a nontrivial solution $w$ of
\begin{equation}\label{auxiliar equation critical}
\begin{cases}
-\Delta w= |x|^\alpha g(x,w) \quad \text{in} \quad B,\\
 w=0 \quad \mbox{on} \quad \partial B,
\end{cases}
\end{equation}
where $g(x,s)=(\widetilde{v}_{\alpha,{\rm rad}, \lambda}+s^++a)^{2^*_\alpha-1}-(\widetilde{v}_{\alpha,{\rm rad}, \lambda}+a)^{2^*_\alpha-1}$, since $w + \widetilde{v}_{\alpha,{\rm rad}, \lambda}$ will correspond to a second radial solution of \eqref{equiv1}.  Observe that any nontrivial solution $w$ of \eqref{auxiliar equation critical} satisfies $w >0$ in $B$.

The functional associated to \eqref{auxiliar equation critical} is
$$
J(w)=\frac{1}{2}\int_B|\nabla w|^2-\int_B |x|^\alpha G(x,w)dx, \quad w \in H^1_{0, {\rm rad}}(B),
$$
where
\begin{multline*}
G(x, s) :=\int_0^s g(x,t)\,dt \\
=\frac{(\widetilde{v}_{\alpha,{\rm rad}, \lambda}+s^++a)^{2^*_\alpha}}{2^*_\alpha}-\frac{(\widetilde{v}_{\alpha,{\rm rad}, \lambda}+a)^{2^*_\alpha}}{2^*_\alpha}-(\widetilde{v}_{\alpha,{\rm rad}, \lambda}+a)^{2^*_\alpha-1}s^+  .
\end{multline*}
Then $0$ is a local minimum of $J$ on $H_{0,{\rm rad}}^1(B)$ (see \cite[Lemma 4.2]{abc} for a similar result) and
we are going to prove the existence of a nonzero critical point for $J$.

By contradiction, assume that $0$ is the only critical point of $J$. Then for some ball
$B(0, r)$ in $H^1_{0,{\rm rad}}(B)$
\begin{equation}\label{minstrict}
J(0) < J(w)
\end{equation}
for all $w\in B(0, r) \menos\{0\}$. The following lemma, which is similar to \cite[Lemma 4.3 (ii)]{abc}, will be proved below.

\begin{lemma}\label{subsobolev}
Assume that $0$ is the only critical point of $J$. Then $J$ satisfies the $(PS)_c$ condition
for all levels $c$ with
\begin{equation}\label{eqsubsob}
c < c_0 := S_\alpha^\frac{N+\alpha}{\alpha+2}\left(\frac{\alpha+2}{2(N+\alpha)}\right).
\end{equation}
\end{lemma}

\medbreak 
From the preceding lemma together with \cite[Theorem 5.10]{djairo-LecturesEkeland}, which just requires the $(PS)_c$ condition to
hold at the level of the local minimum (here the level $J(0) = 0 < c_0$), one deduces
from \eqref{minstrict} that
\begin{equation}
J(0)<\inf\{J(w):\,w\in H_{0,{\rm rad}}^1(B)\mbox{ and }\|w\|=\rho\},
\end{equation}
holds for all $\rho$ in $(0,r)$. We intend to apply  the mountain pass theorem. For this purpose we will show the existence of $w_1\in H^1_{0,{\rm rad}}(B)$ such that $\|w_1\|> \rho$, $J(w_1) < 0$ and that the infmax value of $J$ over the family of all continuous paths
from 0 to $w_1$ is smaller than $c_0$. Once this is done, the usual mountain pass theorem yields the
existence of a nonzero critical point for $J$, a contradiction which will complete the proof
of this part of the theorem (Theorem {\rm \ref{teorema multiplicidade introduction} (I)).

To construct a $w_1$ as above, we consider as in \cite{abc} functions of the form $tu_{\alpha, \epsilon}$ with $t > 0$
and
\[
\gd{u_{\alpha, \epsilon}(x): = \frac{\varphi(x)\epsilon^{\frac{N-2}{2(\alpha + 2)}}}{(\epsilon + |x|^{\alpha + 2})^{\frac{N-2}{\alpha + 2}}}}, \quad x \in \R^N,
\]
where $\epsilon > 0$, $\varphi$ is a fixed $C_{c, {\rm rad}}^{\infty}(B)$-function such that $\varphi \equiv 1$ in $B(0,1/2)$ and $0 \leq \varphi \leq 1$. Note that
\[
\gd{  U_{\alpha} (x)= \frac{1}{(1 + |x|^{\alpha + 2})^{\frac{N- 2}{\alpha + 2}}}, \quad x \in \R^N,   }
\]
is a function where $S_{\alpha}$ is attained. Set
\[
\gd{K_1 = \int_{\R^N}|\nabla U_{\alpha}|^2 dy = (N-2)^2 \int_{\R^N} \frac{|y|^{2(\alpha + 1)}}{(1 + |y|^{\alpha + 2})^{\frac{2(N+\alpha)}{\alpha + 2}}} dy}
\]
and
\[ \gd{K_2 = |U_\alpha|_{2^*_{\alpha}, \alpha}^2 = \left(\int_{\R^N} \frac{|y|^{\alpha}}{(1 + |y|^{\alpha + 2})^{\frac{2(N+\alpha)}{\alpha + 2}}}dy\right)^{\frac{N-2}{N+\alpha}}. }
\]
Then observe that $S_{\alpha} =K_1/K_2$.

The following lemma implies that for $\epsilon$ sufficiently small, the infmax
value of $J$ over the family of all continuous paths from 0 to $w_1 = t_\epsilon u_{\alpha,\epsilon}$ is indeed smaller than $c_0$.

\begin{lemma}\label{subc0}
We have that
$$
\sup_{t>0} J(tu_{\alpha,\epsilon}) < c_0
$$
for $\epsilon > 0$ sufficiently small.
\end{lemma}
The above two lemmas, that are proved below, and the fact that \linebreak $\lim_{t \rt \infty} J(t u_{\alpha, \epsilon}) = - \infty$ complete the proof of Theorem \ref{teorema multiplicidade introduction} (I).

\begin{proof}[\textbf{Proof of Lemma {\rm \ref{subsobolev}}}]
Let $w_n$ be a $(PS)_c$ sequence with $c < c_0$, i.e.
\begin{equation}\label{ps1}
\frac{1}{2}\|w_n\|^2-\int_B |x|^\alpha G(x,w_n)dx\longrightarrow c
\end{equation}
\begin{equation}\label{ps2}
\left|\int_B\nabla w_n\nabla\phi\,dx-\int_B |x|^\alpha g(x,w_n)\phi dx \right|\leq\varepsilon_n\|\phi\|,\qquad\forall\,\phi\in H_{0,{\rm rad}}^1(B),
\end{equation}
where $\varepsilon_n\to 0$. We first observe that $w_n$ remains bounded in $H_{0,{\rm rad}}^1(B)$. This follows by
multiplying \eqref{ps2}  with $\phi = \widetilde{v}_{\alpha,{\rm rad}, \lambda} +w_n$ by $1/2^*_\alpha$ and subtracting from \eqref{ps1}; the terms
of power $2^*_{\alpha}$ cancel and the remaining dominating term is $\|w_n\|^2$, which yields
the boundedness of $w_n$. So, for a subsequence, $w_n\rightharpoonup w_0$ in $H^1_{0,{\rm rad}}(B)$ and $w_n \rightarrow w_0$ in $L^r(B,|x|^\alpha)$ for
any $1 \leq r < 2_\alpha^*$. Then $w_0$ solves
$$
\begin{cases}
-\Delta w_0=|x|^\alpha g(x,w_0) \quad  \mbox{in }B\\ w_0=0 \quad \mbox{on }\partial B,
\end{cases}
$$
and consequently, by the assumption of this lemma, $w_0 = 0$. We now go back to \eqref{ps2}
with $\phi = w_0 + w_n$, multiply again by $1/2^*_\alpha$ and subtract from \eqref{ps1} to get
\begin{equation}\label{ps3}
\lim_{n\to\infty}\|w_n\|^2=2c\left(\frac{N+\alpha}{\alpha+2}\right).
\end{equation}
There are two possibilities: either $c = 0$ or $c \ne 0$. If $c = 0$ then $w_n \rt w_0$ in $H^1_{0,{\rm rad}}(B)$
by \eqref{ps3} and we are done. We will now see that $c \ne0$ leads to a contradiction. For that
purpose we deduce from \eqref{ps2} with $\phi = w_n$ that
\begin{equation}\label{ps4}
\lim_{n\to\infty}\|w_n\|^2=\lim_{n\to\infty}\int_B|x|^\alpha g(x,w_n)w_n\,dx=\lim_{n\to\infty}\int_B |x|^\alpha(w_n^+)^{2^*_\alpha}dx.
\end{equation}
By definition of $S_\alpha$, we have that
\begin{equation}\label{ps5}
\|w_n\|^2\geq S_\alpha\left(\int_B |x|^\alpha|w_n|^{2_\alpha^*dx}\right)^{2/2_\alpha^*}\geq S_\alpha\left(\int_B|x|^\alpha(w_n^+)^{2_\alpha^*}\,dx\right)^{2/2^*_\alpha}.
\end{equation}
It follows from \eqref{ps3}--\eqref{ps5} that
$$
2c\left(\frac{N+\alpha}{\alpha+2}\right)\geq S_\alpha\left(2c\left(\frac{N+\alpha}{\alpha+2}\right)\right)^{2/2_\alpha^*},
$$
ant then $c\geq c_0$. This contradicts \eqref{eqsubsob} and completes the proof of Lemma \ref{subsobolev}. 
\end{proof}

\begin{proof}[\textbf{Proof of Lemma {\rm \ref{subc0}}}]
We have that
\begin{multline}\label{majoracaolinear}
g(x,s)=(\widetilde{v}_{\alpha,{\rm rad}, \lambda}+s^++a)^{2^*_\alpha-1}-(\widetilde{v}_{\alpha,{\rm rad}, \lambda}+a)^{2^*_\alpha-1}\\\geq(s^+)^{2^*_\alpha-1}+(2^*_\alpha-1)(\widetilde{v}_{\alpha,{\rm rad}, \lambda}+a)^{2^*_\alpha-2}s^+.
\end{multline}
Consequently,
$$
J(tu_{\alpha,\epsilon})\leq\frac{t^2\|u_{\alpha,\epsilon}\|^2}{2}-\frac{t^{2^*_\alpha}}{2_\alpha^*}\int_B |x|^\alpha u_{\alpha,\epsilon}^{2_\alpha^*}-\frac{t^2}{2}(2^*_\alpha-1)\int_B|x|^\alpha (\widetilde{v}_{\alpha,{\rm rad}, \lambda}+a)^{2^*_\alpha-2}u_{\alpha,\epsilon}^2.
$$
Since $\widetilde{v}_{\alpha,{\rm rad}, \lambda}\geq b_0>0$ on the support of $u_{\alpha,\epsilon}^2$, we deduce
$$
J(tu_{\alpha,\epsilon})\leq\frac{t^2\|u_{\alpha,\epsilon}\|^2}{2}-\frac{t^{2^*_\alpha}}{2_\alpha^*}\int_B|x|^\alpha  u_{\alpha,\epsilon}^{2_\alpha^*}-\frac{t^2}{2}(2^*_\alpha-1)(b_0+a)^{2^*_\alpha-2}\int_B |x|^\alpha u_{\alpha,\epsilon}^2.
$$

By definition we have
\[
\gd{\nabla u_{\alpha, \epsilon}(x) =\epsilon^{\frac{N-2}{2(\alpha + 2)}}\left[ \frac{\nabla \varphi(x)}{(\epsilon + |x|^{\alpha + 2})^{\frac{N-2}{\alpha + 2}}} - (N-2) \frac{\varphi(x)|x|^{\alpha}x}{(\epsilon + |x|^{\alpha + 2})^{\frac{N+\alpha}{\alpha + 2}}}\right]. }
\]
Then, since $N\geq 3$ and $\varphi \equiv 1$ in $B(0,1/2)$, we get
\[
\gd{ \int_B |\nabla u_{\alpha, \epsilon}|^2 dx =\epsilon^{\frac{N-2}{\alpha + 2}}\left[ O(1) + (N-2)^2 \int_{\R^N} \frac{|x|^{2(\alpha + 1)}}{(\epsilon + |x|^{\alpha + 2})^{\frac{2(N + \alpha)}{\alpha + 2}}}dx\right] }
\]
and so
\begin{equation}\label{estimate gradient}
\gd{ \int_B |\nabla u_{\alpha, \epsilon}|^2 dx = \epsilon^{\frac{N-2}{\alpha + 2}}\left[O(1) + K_1 \epsilon^{- \frac{N-2}{\alpha + 2}}\right]=O(\epsilon^{\frac{N-2}{\alpha + 2}})+K_1.  }
\end{equation}

Now, since $\varphi \equiv 1$ in $B(0,1/2)$, we have
\begin{equation}\label{estimate critico}
\begin{split}
\int_B |u_{\alpha, \epsilon}|^{2^*_{\alpha}} |x|^{\alpha} dx  &= \epsilon^{\frac{N+\alpha}{\alpha + 2}}\int_B \frac{\varphi^{2^*_{\alpha}}(x) |x|^{\alpha}}{(\epsilon + |x|^{\alpha + 2})^{\frac{2(N + \alpha)}{\alpha + 2}}}dx\\ & = \epsilon^{\frac{N+\alpha}{\alpha + 2}}O(1) + \epsilon^{\frac{N+\alpha}{\alpha + 2}}\int_{\R^N} \frac{|x|^{\alpha}}{(\epsilon + |x|^{\alpha + 2})^{\frac{2(N + \alpha)}{\alpha + 2} }}dx \\
& = O(\epsilon^{\frac{N+\alpha}{\alpha + 2}}) + K_2'
\end{split}
\end{equation}
where $K_2'= K_2^{2^*_{\alpha}/2}$. Then observe that $\left(O(\epsilon^{\frac{N+\alpha}{\alpha + 2}}) + K_2'\right)^{2/2^*_{\alpha}} = K_2 + O(\epsilon^\frac{N+\alpha}{\alpha+2})$.

Now we estimate $\int_B|u_{\alpha, \epsilon}|^2|x|^{\alpha} dx$. We prove that
\begin{equation}\label{novaeq}
 \int_B |u_{\alpha, \epsilon}|^2 |x|^{\alpha} dx =
 \left\{
\begin{array}{l}
\gd{K_ 3 \, \epsilon + O(\epsilon^{\frac{N-2}{\alpha+ 2}}), \quad \text{if} \quad 0 < \alpha < N-4,} \\
\gd{ K_ 3 \, \epsilon \,| \log \epsilon| + O(\epsilon), \quad 0 < \alpha = N-4,} \\
\gd{  K_ 3 \, \epsilon^{\frac{N-2}{\alpha+2}}  + o(\epsilon^{\frac{N-2}{\alpha+2}}), \quad N-4 < \alpha},
\end{array}
\right.
\end{equation}
with, respectively, 
\[
K_3 = \int_{\R^N} \frac{|x|^{\alpha}}{(1 + |x|^{\alpha + 2})^{\frac{2(N-2)}{\alpha + 2}}}dx, \ \ K_3 = \frac{\omega_{N}}{N-2}, \ \ K_3 = \int_B \frac{\varphi^2(x) dx}{|x|^{2(N-2)-\alpha}} dx.
\] 
Indeed, since $\varphi\equiv 1$ in $B(0,1/2)$ we have
\begin{equation}\label{l2 geral}
\begin{split}
\int_B |u_{\alpha, \epsilon}|^2|x|^{\alpha} dx &=\epsilon^{\frac{N-2}{\alpha + 2}}\int_B \frac{\varphi^2(x) |x|^{\alpha}}{(\epsilon + |x|^{\alpha + 2})^{\frac{2(N-2)}{\alpha + 2}}} dx \\
&= \epsilon^{\frac{N-2}{\alpha + 2}}\left(O(1) + \int_B \frac{ |x|^{\alpha}}{(\epsilon + |x|^{\alpha + 2})^{\frac{2(N-2)}{\alpha + 2}}} dx\right).
\end{split}
\end{equation}

\medbreak

\noindent \textbf{Case $0\leq \alpha < N-4$.} In this case, from \eqref{l2 geral}, we have
\[
\begin{split}
\int_B |u_{\alpha, \epsilon}|^2|x|^{\alpha} dx &= O(\epsilon^{\frac{N-2}{\alpha + 2}}) +  \epsilon^{\frac{N-2}{\alpha + 2}}\int_{\R^N} \frac{ |x|^{\alpha}}{(\epsilon + |x|^{\alpha + 2})^{\frac{2(N-2)}{\alpha + 2}}} dx \\
& = O(\epsilon^{ \frac{N-2}{\alpha + 2}}) + \epsilon K_3 \, ,
\end{split}
\]
with $\gd{K_3 = \int_{\R^N} \frac{|x|^{\alpha}}{(1 + |x|^{\alpha + 2})^{\frac{2(N-2)}{\alpha + 2}}}dx}$. Then, from \eqref{estimate gradient} and \eqref{estimate critico}, we obtain
$$
\sup_{t\geq 0}J(tu_{\alpha, \epsilon})\leq \left(\frac{\alpha+2}{2(N+\alpha)}\right)(S_\alpha^\frac{N+\alpha}{\alpha+2} + O (\epsilon^{\frac{N-2}{\alpha + 2}}) - K_4 \epsilon),
$$
for some positive constant $K_4$, and we are done since $0 \leq \alpha < N-4$.
\medbreak

\noindent \textbf{Case $0 \leq \alpha = N-4$.} First, we have
\begin{equation}\label{l2 partial}
\begin{split}
\gd{\int_B \frac{ |x|^{\alpha}}{(\epsilon + |x|^{\alpha + 2})^{\frac{2(N-2)}{\alpha + 2}}} dx } &= \gd{\omega_{N} \int_0^1 \frac{r^{\alpha +N -1}}{(\epsilon + r^{\alpha + 2})^{\frac{2(N-2)}{\alpha + 2}}} dr } \\
&= \gd{\frac{\omega_{N}}{\alpha + 2}} \int_{\epsilon}^{1 + \epsilon} \left( \frac{z- \epsilon}{z^2} \right)^{\frac{N-2}{\alpha + 2}} dz.
\end{split}
\end{equation}
From \eqref{l2 partial} and using that $\alpha = N-4$ we get
\[
\gd{\int_B \frac{ |x|^{\alpha}}{(\epsilon + |x|^{\alpha + 2})^{\frac{2(N-2)}{\alpha + 2}}} dx  = \frac{\omega_{N}}{\alpha + 2} \int_{\epsilon}^{1 + \epsilon} \frac{z-\epsilon}{z^2} dz  = O(1) + \frac{\omega_{N}}{N-2} \,|\log \epsilon|}.
\]
Then,
\[
\gd{\int_B |u_{\alpha, \epsilon}|^2|x|^{\alpha} dx = O(\epsilon) + \frac{\omega_{N}}{N-2} \,\epsilon|\log \epsilon| ,}
\]
and at this time we have
$$
\sup_{t\geq 0}J(tu_{\alpha, \epsilon})\leq \left(\frac{\alpha+2}{2(N+\alpha)}\right)(S_\alpha^\frac{N+\alpha}{\alpha+2} + O (\epsilon) - K_4\epsilon|\log\epsilon|),
$$
for some positive constant $K_4$ and we are done.
\medbreak

\noindent\textbf{Case $N-4 < \alpha$.} In this case we can apply the dominated convergence theorem to obtain that
\[
\int_B \frac{\varphi^2(x) |x|^{\alpha}}{(\epsilon + |x|^{\alpha + 2})^{\frac{2(N-2)}{\alpha + 2}}} dx = \int_B \frac{\varphi^2(x) dx}{|x|^{2(N-2)-\alpha}} dx + o(1) 
\]
and then that
\begin{multline*}
\int_B |u_{\alpha, \epsilon}|^2|x|^{\alpha} dx =\epsilon^{\frac{N-2}{\alpha + 2}}\int_B \frac{\varphi^2(x) |x|^{\alpha}}{(\epsilon + |x|^{\alpha + 2})^{\frac{2(N-2)}{\alpha + 2}}} dx \\
= o(\epsilon^{\frac{N-2}{\alpha + 2}}) + \epsilon^{\frac{N-2}{\alpha + 2}}\int_B \frac{\varphi^2(x) dx}{|x|^{2(N-2)-\alpha}} dx.
\end{multline*}
In this case the inequality \eqref{majoracaolinear} is not suitable for our purposes. We emphasize that the inequality $N<\alpha +4$ corresponds to critical dimensions associated to problem \eqref{problem weight introduction}; compare with the critical dimension $N=3$ ($\alpha =0$) in the paper of Brezis and Nirenberg \cite{brezis-nirenberg}. Instead of \eqref{majoracaolinear} we use that
\begin{multline}\label{majoracaosuperlinear}
g(x,s)=(\widetilde{v}_{\alpha,{\rm rad}, \lambda}+s^++a)^{2^*_\alpha-1}-(\widetilde{v}_{\alpha,{\rm rad}, \lambda}+a)^{2^*_\alpha-1}\\\geq(s^+)^{2^*_\alpha-1}+(2^*_\alpha-1)(\widetilde{v}_{\alpha,{\rm rad}, \lambda}+a)(s^+)^{2^*_\alpha-2}.
\end{multline}
Consequently,
$$
J(tu_{\alpha,\epsilon})\leq\frac{t^2\|u_{\alpha,\epsilon}\|^2}{2}-\frac{t^{2^*_\alpha}}{2_\alpha^*}\int_B |x|^\alpha u_{\alpha,\epsilon}^{2_\alpha^*}-t^{2^*_\alpha-1}\int_B|x|^\alpha (\widetilde{v}_{\alpha,{\rm rad}, \lambda}+a)u_{\alpha,\epsilon}^{2^*_\alpha-1}
$$
and observe that $2^*_{\alpha} -2 < 2^*_{\alpha} -1 < 2^*_{\alpha}$; compare with \cite[eq. (0.6)]{brezis-nirenberg}. Since $\widetilde{v}_{\alpha,{\rm rad}, \lambda}\geq b_0>0$ on the support of $u_{\alpha,\epsilon}^2$, we deduce that
$$
J(tu_{\alpha,\epsilon})\leq\frac{t^2\|u_{\alpha,\epsilon}\|^2}{2}-\frac{t^{2^*_\alpha}}{2_\alpha^*}\int_B|x|^\alpha  u_{\alpha,\epsilon}^{2_\alpha^*}-\frac{t^2}{2}(2^*_\alpha-1)(b_0+a)\int_B |x|^\alpha u_{\alpha,\epsilon}^{2^*_\alpha-1}.
$$
Now, since $\varphi \equiv 1$ in $B(0,1/2)$, we have
\begin{multline}\label{estimate critico-1}
\int_B |u_{\alpha, \epsilon}|^{2^*_{\alpha}-1} |x|^{\alpha} dx  = \epsilon^{\frac{N+2(1+ \alpha)}{2(\alpha + 2)}}\int_{\R^N} \frac{\varphi^{2^*_{\alpha}-1}(x) |x|^{\alpha}}{(\epsilon + |x|^{\alpha + 2})^{\frac{N + 2(1+\alpha)}{\alpha + 2}}}dx\\ 
 = \epsilon^{\frac{N+2(1+\alpha)}{\alpha + 2}}\left[ O(1) + \int_{\R^N} \frac{|x|^{\alpha}}{(\epsilon + |x|^{\alpha + 2})^{\frac{N + 2(1+\alpha)}{\alpha + 2} }}dx\right] \\
 =\epsilon^{\frac{N+2(1+\alpha)}{\alpha + 2}} \left[ O(1) + \epsilon^{-1}\int_{\R^N} \frac{|x|^{\alpha}}{(1 + |x|^{\alpha+2})^{\frac{N+2(1+\alpha)}{\alpha + 2}}}dx\right] = O(\epsilon^{\frac{N+2(1+\alpha)}{\alpha + 2}}) + K_5 \epsilon^{\frac{N-2}{2(\alpha+2)}}.
\end{multline}
with $K_5 = \gd{\int_{\R^N} \frac{|x|^{\alpha}}{(1 + |x|^{\alpha+2})^{\frac{N+2(1+\alpha)}{\alpha + 2}}}}dx$. Then from \eqref{estimate gradient} and \eqref{estimate critico} we infer that
$$
\sup_{t\geq 0}J(tu_{\alpha, \epsilon})\leq \left(\frac{\alpha+2}{2(N+\alpha)}\right)(S_\alpha^\frac{N+\alpha}{\alpha+2} + O (\epsilon^{\frac{N-2}{\alpha+2}}) - K_4 \epsilon^{\frac{N-2}{2(\alpha+2)}}),
$$
for some positive constant $K_4$ and we are done.
%
%Since $\alpha > N-4$ we get that $\frac{\varphi^2(x)}{|x|^{2(N-2)-\alpha}} \in L^1(B)$ and so from the first identity of \eqref{l2 geral} and from the dominated convergence theorem we infer that
%\[
%\int_B |u_{\alpha, \epsilon}|^2|x|^{\alpha} dx = \epsilon^{\frac{N-2}{\alpha+2}} \left( \int_B \frac{\varphi^2(x) dx}{|x|^{2(N-2)-\alpha}} dx +  o(1) \right).
%\]
%Then we have, \textcolor{blue}{corrigir aqui para baixo...problema com mesma ordem.}
%$$
%\sup_{t\geq 0}J(tu_{\alpha, \epsilon})\leq \left(\frac{\alpha+2}{2(N+\alpha)}\right)(S_\alpha^\frac{N+\alpha}{\alpha+2} + O (\epsilon^{\frac{N-2}{\alpha+2}}) - K_4\epsilon^{\frac{N-2}{\alpha+2}}|\log\epsilon|),
%$$
%and we are done. 
\end{proof}

\begin{remark}
In the proof of Lemma {\rm \ref{subc0}} we could have used \eqref{majoracaosuperlinear} in all the cases $0 <\alpha < N-4$, $\alpha = N-4$ and $\alpha > N-4$. However, we decided also to use \eqref{majoracaolinear} to emphasize the critical dimensions $N \in [3, \alpha+4)$ associated to the equation \eqref{problem weight introduction}.
\end{remark}

%xxxxxxxxxxxxxxxxxxxxxxxxxxxxxxxxxxxxxxxx

%We denote by $c_{\alpha, \lambda, {\rm rad}}$ the critical value of $J_{{\rm rad}}$ associated to the local minimum $\widetilde{v}_{\alpha,{\rm rad}, \lambda}$ given by Proposition \ref{prop energia negativa}.
%
%\begin{lemma}\label{location ps condition radial}
% $J_{{\rm rad}}$ satisfies the $(PS)_c$ condition
%for all levels $c$ with
%\begin{equation}\label{eqsubsob}
%c < c_{\alpha, \lambda, {\rm rad}} + S_\alpha^\frac{N+\alpha}{\alpha+2}\left(\frac{\alpha+2}{2(N+\alpha)}\right).
%\end{equation}
%\end{lemma}
%
%
%\begin{lemma}\label{mountain pass radial below ps}
%Let $0< \lambda < \lambda_{*}$. Then
%\[
%m_{\alpha, \lambda, {\rm rad}} < c_{\alpha, \lambda, {\rm rad}} + S_\alpha^\frac{N+\alpha}{\alpha+2}\left(\frac{\alpha+2}{2(N+\alpha)}\right).
%\]
%\end{lemma}
%
%
%\textcolor{blue}{we need to do this.}
\subsection{On the existence of a radial mountain pass solution to (\ref{problem weight introduction})} \label{secaoMP}
In Section \ref{duassol} we proved that \eqref{problem weight introduction} has two radial solutions for every $0 < \lambda < \lambda_*$ and for every $1 <p \leq 2^*_{\alpha} -1$. Here we guarantee the existence of a mountain pass solution in the case that $\lambda >0$ is sufficiently small.

\begin{proposition}\label{lemmaMP}
Let $1< p$ and in case $N\geq 3$ also assume $p < 2^*_{\alpha} -1$. Then there exists $\lambda_{0} = \lambda_{0}(N,p) \in (0, \lambda_*)$, such that for every $0 < \lambda < \lambda_{0}$ the functional $J_{\lambda, {\rm rad}}$ has a mountain pass solution associated to its local minimum $\widetilde{v}_{\alpha,{\rm rad}, \lambda}$, with $\widetilde{v}_{\alpha,{\rm rad}, \lambda}$ as in Proposition \ref{prop energia negativa}. We emphasize that $\lambda_{0}(N,p)$ does not depend on $\alpha$.
\end{proposition}
\begin{proof}
First we recall that
\[
e_\alpha < e_0 \quad \text{in} \quad B \quad \forall \ \alpha >0.
\]
From Proposition \ref{prop energia negativa}, more precisely from \eqref{barra comportamento} (cf. Lemma \ref{upper zero}), we infer that there exists $\lambda_0 = \lambda_0(N,p)<1$ such that
\[
\widetilde{v}_{\alpha,{\rm rad}, \lambda} \leq 2 \lambda^{p/(p-1)} e_0 \quad \forall \ \lambda  \in (0, \lambda_0).
\]
Then, taking into account that $a = \lambda^{1/(p-1)}$, we get that
\begin{multline}\label{dentro}
\int_B |\nabla \widetilde{v}_{\alpha,{\rm rad}, \lambda}|^2 dx  = \int_B |x|^{\alpha} (\widetilde{v}_{\alpha,{\rm rad}, \lambda} + a)^{p}\widetilde{v}_{\alpha,{\rm rad}, \lambda}  dx \\ \leq \int_B  |\widetilde{v}_{\alpha,{\rm rad}, \lambda} + a|^{p+1} dx \leq C \lambda^{(p+1)/(p-1)} \quad \forall \ \lambda  \in (0, \lambda_0).
\end{multline}

Now observe that
\begin{multline}\label{independente2}
|J_{\lambda, {\rm rad}}(v) - J_{0, {\rm rad}}(v)| = \frac{1}{p+1}\left| \int_B |x|^{\alpha} |v+ a|^{p+1} - |x|^{\alpha}|v|^{p+1} dx \right| \\ \leq 2^{p-1} a \int_B |x|^{\alpha}(|v|^p + a^{p}) dx \leq 2^{p-1} \lambda^{1/(p-1)} \left( C \|v\|^p + \lambda^{p/(p-1)} |B|\right) 
\end{multline}
with $C= C(N,p)$ is such that
\begin{equation}\label{indepalpha}
\int_B |u|^{p}|x|^{\alpha} dx \leq C(N,p) \left(\int_B |\nabla u |^2 dx\right)^{p/2}, \quad \forall \, u \in  H^1_{0,\rm{rad}}(B).
\end{equation}
Observe that $C(N,p)$ may be taken independently of $\alpha$. Indeed, for every $u \in H^1_{0,\rm{rad}}(B)$, we have from Ni's pointwise estimate \eqref{des ni regularidade} that
\begin{multline*}
\left(\int_B |u|^{2^*_{\alpha}} |x|^{\alpha} dx\right)^{1/2^*_{\alpha}} = \left(\int_B |u|^{2^{*}} |u|^{\frac{2\alpha}{N-2}}|x|^{\alpha}\right)^{1/2^*_{\alpha}} dx \\ \leq \left( \| \nabla u \|_2^{\frac{2\alpha}{N-2}}\frac{1}{(\omega_{N-1}(N-2))^{\frac{\alpha}{N-2}}} \int_B |u|^{2^*} dx \right)^{1/2^*_{\alpha}}\\
\leq \left( \| \nabla u \|_2^{2^*_{\alpha}} \frac{1}{(\omega_{N-1}(N-2))^{\frac{\alpha}{N-2}}} S^{\frac{N}{N-2}}\right)^{1/2^*_{\alpha}} = \frac{S^{\frac{N}{2(N+\alpha)}}}{\left( (N-2) \omega_{N-1} \right)^\frac{\alpha}{2(N+\alpha)}} \left(\int_B |\nabla u|^2 dx \right)^{1/2}.
\end{multline*}
Then take into account that the constants for the embeddings $L^{2^*_{\alpha}}(B, |x|^{\alpha}) \hookrightarrow L^{p}(B, |x|^{\alpha})$ can be bounded from above uniformly with respect to $\alpha$.
%In the case that $1< p \leq 2^{*}$, in particular with $1< p \leq 2^{*}-1$ then $C$ depends only on $N$ and $p$ since we can use that
%\[
%\int_B |x|^{\alpha} |v|^p dx \leq \int_B |v|^p dx.
%\]

Then, we recall that 
\[
J_{0, {\rm rad}}(v) = \frac{1}{2} \int_B |\nabla v|^2 dx  - \frac{1}{p+1}\int_B |x|^{\alpha} |v|^{p+1} dx
\]
has a strict minimum at $v=0$. Moreover, using again \eqref{indepalpha}, now with $p+1$ in place of $p$, there exists $\epsilon(N, p)> 0$ and $r = r(N,p)>0$ such that
\[
J_{0, {\rm rad}}(v) \geq \epsilon \quad \forall \ v \in H^1_{0, \rm{rad}}(B) \ \ \text{s.t.} \ \ \|v \| = r. 
\]
Hence, combining the last inequality with \eqref{dentro} and \eqref{independente2}, there exists $0< \widetilde{\lambda}_{0} \leq \lambda_0$,  with $\widetilde{\lambda}_{0} = \widetilde{\lambda}_{0}(N,p)$, such that for every $0 < \lambda < \widetilde{\lambda}_{0}$:
\begin{align*}
&\|\widetilde{v}_{\alpha,{\rm rad}, \lambda}\| < \frac{r}{2}, \quad J_{\lambda, {\rm rad}}(\widetilde{v}_{\alpha,{\rm rad}, \lambda}) <0, \ \ \text{and}\\
&J_{\lambda, {\rm rad}}(v) > \frac{\epsilon}{2} \quad \forall \ v \in H^1_{0, \rm{rad}}(B) \ \ \text{s.t.} \ \ \|v \| = r.
\end{align*}
It is also clear that $J_{\lambda, {\rm rad}}(R \, \widetilde{v}_{\alpha,{\rm rad}, \lambda})< 0$ for $R> 0$ sufficiently large. Therefore, since $J_{\lambda, {\rm rad}}$ satisfies the $(PS)_c$ condition at every level $c$, we apply the standard version of the mountain pass lemma \cite{ambrosetti-rabinowitz}. \qedhere
\end{proof}

\subsection{Partially symmetric mountain pass solutions}
Assume all the hypotheses from Proposition \ref{prop energia negativa l}. So,

\[
J_{\lambda, l}(v) = \int_B |\nabla u |^2 dx - \frac{1}{p+1} \int_B |x|^{\alpha} | v + a|^{p+1} dx, \quad v \in H_{l}(B), \quad a^{p-1} = \lambda,
\]
has a local minimum at $\widetilde{v}_{\alpha, l , \lambda}$ and a mountain pass solution in the case that $0 < \lambda < \lambda_0(N,p)$; cf. Section \ref{secaoMP}. At this time, to see that $\lambda_0(N,p)$ can be taken independent of $\alpha$ and $l$ we refer to \cite[Corollary 2.3]{badiale-serra}. As before, we prove that any mountain pass solution of $J_{\lambda, l}$ associated to the mountain pass level
\[
\begin{array}{l}
\gd{m_{\alpha,\lambda,l} = \inf_{\gamma \in \Gamma_{\alpha, \lambda, l}} \max_{t \in [0,1]} J_{l}(\gamma(t))},\\\\
\gd{\Gamma_{\alpha, \lambda, l} = \{ \gamma \in C([0,1], H_{l}(B)); \gamma(0) =\widetilde{v}_{\alpha, l , \lambda} \,\, \text{and} \,\, \gamma(1) = v_0\}},
\end{array}
\]
is positive in $B$.  Take $V_{\alpha, \lambda, l}$ a mountain pass critical point of $J_{l}$ associated to the mountain pass level $m_{\alpha,a, l}$.

%\begin{proof}[Proof of Theorem {\rm \ref{teorema multiplicidade} (I)}: case $N \geq 3$ and $p= 2^*_{\alpha}-1$] ${}$
%
%\nin Estudar o nível do passo da montanha. Lema de compacidade e etc.....o procedimento será identico o de cima após mostrar que o nível mountain pass está abaixo do nível de compacidade.
%\end{proof}

\subsection{Solutions in the space $H^1_0(B)$}
Here we suppose $1 < p < 2^* -1$. So, by Proposition \ref{prop energia negativa l},
\[
J_{\lambda}(v) = \int_B |\nabla u |^2 dx - \frac{1}{p+1} \int_B |x|^{\alpha} | v + a|^{p+1} dx, \quad v \in H^1_{0}(B),\quad a^{p-1} = \lambda,
\]
has a local minimum at $\widetilde{v}_{\alpha, \lambda}$ and a mountain pass solution in the case that $0< \lambda < \lambda_0(N,p)$; cf. Section \ref{secaoMP}.  At this time, to see that $\lambda_0(N,p)$ can be taken independent of $\alpha$ we can use that
\[
\int_B |v|^{p+1} |x|^{\alpha} dx \leq \int_B |v|^{p+1} dx
\]
and we use the classical Sobolev embeddings. As before, we prove that any mountain pass solution of $J_{\lambda}$ associated to the mountain pass level
\[
\begin{array}{l}
\gd{m_{\alpha, \lambda} = \inf_{\gamma \in \Gamma_{\alpha, \lambda}} \max_{t \in [0,1]} J(\gamma(t))},\\\\
\gd{\Gamma_{\alpha, \lambda} = \{ \gamma \in C([0,1], H^1_{0}(B)); \gamma(0) =\widetilde{v}_{\alpha, \lambda} \,\, \text{and} \,\, \gamma(1) = v_0\}},
\end{array}
\]
is positive in $B$.  Take $V_{\alpha, \lambda}$ a mountain pass critical point of $J$ associated to the mountain pass level $m_{\alpha,\lambda}$.

In case $N \geq 2$, arguing as in \cite[Proposition 3.1]{squassina-vanschaftingen}, we can prove that for each closed half-space $H$ in $\R^N$, the polarized function
\[
V_{\alpha, \lambda}^{H} = \left\{
\begin{array}{l}
\max\{ u, u \circ \sigma_H \} \quad \text{on} \quad H \cap B,\\
\min\{ u, u \circ \sigma_H \} \quad \text{on} \quad \text (\R^N \menos H)  \cap B,
\end{array}
\right.
\]
is also a solution of \eqref{equiv1} associated to the critical level $m_{\alpha, \lambda}$. Then we argue as in \cite[Lemmas 16 and 17]{berchio-gazzola-weth} to prove that $V_{\alpha, \lambda}$ is Schwarz foliated symmetric indeed.

\subsection{Proof of Theorem \ref{teorema multiplicidade introduction} (II) and (III)}

\begin{proposition}\label{teorema de multiplicidade niveis}
Let $N \geq 1$, $\alpha >0$ and $0 < \gamma < \min\{1, \alpha\}$.
\begin{enumerate}[(i)]
\item If $1 < p < 2^*_{\alpha} -1$, then
\[
\lim_{\lambda \rt 0^+}  V_{\alpha, \lambda,{\rm rad}} = V_{\alpha, {\rm rad}} \quad \text{in} \quad C^{2,\gamma}(\overline{B}).
\]
\item If $1 < p < 2^* -1$, then
\[
 \lim_{\lambda \rt 0^+}  V_{\alpha,\lambda,l} = V_{\alpha, l} \quad \text{and} \quad
 \lim_{\lambda \rt 0} V_{\alpha, \lambda} = V_{\alpha} \quad \text{in} \quad C^{2,\gamma}(\overline{B}).
 \]
\end{enumerate}
Here $V_{\alpha, {\rm rad}}$, $V_{\alpha,l}$ and $V_{\alpha}$ are mountain pass critical points of $J_{0, {\rm rad}}:H^1_{0, {\rm rad}}(B)$, $J_{0, l}:  H_{l}(B) \rt \R$ and $J_0: H^1_0(B) \rt \R$ respectively, associated to the problem
 \[
 -\Delta Z =  |x|^{\alpha} |Z|^{p-1} Z \quad \text{in} \quad B, \quad Z = 0 \quad \text{on} \quad \partial B.
 \]
 We also have the respective convergence of the mountain pass levels.
 \end{proposition}
\begin{proof}
Here we can closely follow the proof of \cite[Theorem 2]{gazzola} and \cite[Theorem 6]{gazzola-malchiodi}, which are based on a priori estimates for positive solution of \eqref{equiv1} for $0 \leq \lambda < \lambda_*$. Since such arguments are indeed very similar to those in the proof of  \cite[Theorem 2]{gazzola} and \cite[Theorem 6]{gazzola-malchiodi}, we omit them here. 

We stress that, in the case (i), the a priori estimates for radial positive solutions of \eqref{equiv1} with $1 < p < 2^{*}_{\alpha}-1$ follows from \cite[p. 2529, Case 2]{phan-souplet}; observe that $P_k=0$, for every $k$, for radial solutions. The a priori estimates in the case (ii) (for all positive solutions of \eqref{equiv1}) is presented in \cite[Theorem 1.3]{phan-souplet}.  Moreover, these a priori estimates depends on $N$, $p$ and $\alpha$. Since these a priori estimates depend on $\alpha$ we see from \eqref{independente2} that the convergence of the mountain pass levels (as $\lambda\rt 0^+$) also depends on $\alpha$.
% We prove lemmas similar to \cite[Lemmas 5, 6, 7, 8]{gazzola} in our framework. In case $1< p< 2^*-1$ we have such a priori estimates given by \cite[Theorem 1.3]{phan-souplet}. Here we proceed as in \cite{gazzola}.
\end{proof}

\begin{proof}[\textbf{Proof of Theorem {\rm \ref{teorema multiplicidade introduction}} (II) and (III) completed}]
For each $N \geq 1$, the mountain pass levels of $V_{\alpha, {\rm rad}}$, $V_{\alpha,l}$ and $V_{\alpha}$ are different, provided $\alpha > \alpha_0(N, p)$; see \cite{smets-su-willem, byeon-wang, badiale-serra}. Then, for every $\alpha > \alpha_0(N, p)$,  we obtain from Propositions \ref{lemmaMP}  and \ref{teorema de multiplicidade niveis} that there exists $\lambda_0 = \lambda_0(N,p,\alpha)$ such that for every $0< \lambda < \lambda_0$ the solutions $V_{\alpha, \lambda, {\rm rad}}$, $V_{\alpha, \lambda,l}$ and $V_{\alpha, \lambda}$ are non rotational equivalent, because they have different positive energy levels. Observe that $\lambda_0$ depends on $\alpha$ because, as we explained in the proof of Proposition \ref{teorema de multiplicidade niveis}, the convergence of the mountain pass levels may depend on $\alpha$. Then, for every $N\geq 1$, $V_{\alpha, \lambda, {\rm rad}}$,  $V_{\alpha, \lambda}$ and $\widetilde{v}_{\alpha, \lambda, {\rm rad}}$, with $\widetilde{v}_{\alpha, \lambda, {\rm rad}}$ from Proposition \ref{prop energia negativa}, produce three non rotational equivalent solutions of \eqref{problem weight introduction}. In addition, $V_{\alpha, \lambda}$ is not radially symmetric, and in case $N\geq 2$, $V_{\alpha, \lambda}$ is Schwarz foliated symmetric.

Using the same arguments, in case $N \geq 4$, $V_{\alpha, \lambda,l}$ and $V_{\alpha, \lambda,j}$ are non rotational equivalent if $j\neq l$; see \cite{badiale-serra, li} for the limit problem with $\lambda=0$. Then we get the existence of at least $[N/2] + 2$ non rotational equivalent solutions for \eqref{problem weight introduction}, since we have $\left[ \frac{N}{2}\right] -1$ choices of $l \in \Z$ such that $2 \leq N - l \leq l$. 
\end{proof}

\begin{remark}\label{remark multiple solutions supercritical}
Assume $N\geq 4$, $l \in \N$, $2 \leq N- l \leq l$ and $2 < p+1 < \frac{2(l+1)}{l-1}$. Then, as we argued in the proof of Theorem {\rm \ref{teorema multiplicidade introduction}} (III), there exist  $\alpha_0(N, p)>0$ and $\lambda_0 = \lambda_0(N,p,\alpha)$ such that for all $\alpha > \alpha_0(N, p)$ and $0< \lambda < \lambda_0$ the solutions $V_{\alpha, \lambda, {\rm rad}}$, $V_{\alpha, \lambda,l}$ have distinct positive critical levels and $\widetilde{v}_{\alpha, \lambda, {\rm rad}}$ has a negative critical level.  Therefore we get the existence of at least three solutions to \eqref{problem weight introduction}. Observe that an estimate like \eqref{independente2}, involving $J_{\lambda,l}$ and $J_{0,l}$, can be used to prove uniformly bound of the mountain pass level, and so of the mountain pass solutions, as $\alpha \rt 0^+$. Finally, observe that the condition $2 < p+1 < \frac{2(l+1)}{l-1}$ includes cases with $p+1 \geq 2^*$. 
\end{remark}

\section{A weighted problem posed in an exterior domain} \label{section related problems} Here we consider $N \geq 3$, $a\geq0$, $\beta \in \R$ and $p>0$ and the problem
\begin{equation}\label{exterior}
 \left\{
\begin{array}{l}
 -\Delta U = \gd{\frac{U^p}{|x|^{\beta}}} \quad \text{in} \quad \R^N \menos \overline{B}, \quad U  >  0 \quad \text{in} \quad \R^N \menos \overline{B},\\
 U  =  a \quad \text{on} \quad \partial B,  \quad U  \rt 0 \quad \text{as} \quad |x| \rt \infty.
\end{array}
 \right.
\end{equation}

If $U : \R^N \menos B \rt \R$, then the Kelvin transform $u(x) = U\left( \frac{x}{|x|^2} \right) |x|^{2-N}$ is such that $u : \overline{B} \menos \{ 0\} \rt \R$ and
\[
\Delta u (x) = |x|^{-N-2}\Delta U \left( \frac{x}{|x|^2} \right).
\]
On the other hand, if $u : \overline{B} \rt \R$ is a $C(\overline{B}) \cap C^2(B)$, then $U(x) = u\left( \frac{x}{|x|^2} \right) |x|^{2-N}$ is such that, $U: \R^N \menos B \rt \R$ is continuous, $U \in C^2(\R^N \menos \overline{B})$ and \linebreak $\lim_{|x| \rt \infty} U(x)|x|^{N-2} = u(0)$.

So, if we search for a solution of \eqref{exterior}, let $u(x) = U\left( \frac{x}{|x|^2} \right) |x|^{2-N}$. Then we are led to study the following problem
\begin{equation}\label{equiv3}
 \left\{
 \begin{array}{l}
 -\Delta u  = | x|^{-N - 2 + \beta + p(N-2)} u^p \quad \text{in} \quad B,\\
 u  = a \quad \text{on} \quad \partial B, \ \  u  > 0 \quad \text{in} \quad B,
\end{array}
 \right.
\end{equation}
which is similar to problem \eqref{boundary inhomogeneous}.
\medbreak

\noindent \textbf{The case with $a=0$.} As a consequence of \cite{ni, smets-su-willem, byeon-wang}:
\begin{enumerate}[(i)]
\item If $\beta \leq 0$, then there exists at least one positive radial $U$ for \eqref{exterior} for all $p > \frac{N + 2 -2\beta}{N-2}$.

\item Consider $0 < \beta \leq \frac{N+2}{2}$. Then there exists at least one positive radial $U$ for \eqref{exterior} for all $p \geq \frac{N + 2 -\beta}{N-2}$ and $p \neq 1$.

\item Consider $\beta > \frac{N+2}{2}$. Then there exists at least one positive radial solution $U$ for \eqref{exterior} for all $p> 0$ and $p\neq 1$.
 \item Consider $N=1,2$, $p> 1$ and $\beta > 0$ large. Then \eqref{exterior} has a least two non rotational equivalent solutions.
\item Consider $N \geq 3$, $1 < p< 2^* -1$ and $\beta> 0$ large. Then \eqref{exterior} has at least $\left[ \frac{N}{2}\right] + 1$ non rotational equivalent solutions.
\end{enumerate}

\medbreak
\noindent \textbf{The case with $a>0$.}
\begin{enumerate}[(i)]
\item In case $p> 1$ and $\beta \geq N+2 - p (N-2)$. Then \eqref{exterior} has a solution if, and only if, $a$ is suitably small; see Theorem \ref{teorema existencia introduction}.
\item In case $\beta \leq 0$ and $p > \frac{N + 2 -2\beta}{N-2}$, then \eqref{exterior} has at least two radial solutions in case $a>0$ is suitably small; see Theorem \ref{teorema multiplicidade introduction} (I).
\item In case $\beta> 0$ is large, $a>0$ is suitably small and $1 < p < 2^* -1$,  then Theorem \ref{teorema multiplicidade introduction} (II) and (III) apply to prove the existence of multiple positive solutions to \eqref{exterior}.
\end{enumerate}

%\section{Footnotes}
%
%Footnotes should be numbered sequentially in superscript
%lowercase roman letters.\footnote{Footnotes should be
%typeset in 8 pt roman at the bottom of the page.}

\section*{Acknowledgments}

\noindent Leonelo Iturriaga has been partially supported by  Programa Basal PFB 03, CMM,
U. de Chile; Fondecyt grant 1120842 and USM grant No. 12.12.11. Ederson Moreira dos Santos has been partially supported by CNPq \#309291/2012-7 grant and FAPESP \#10/19320-7 grant. Pedro Ubilla has been partially supported Fondecyt grant 1120524.
\appendix

%\section{Appendices}
%
%Appendices should be used only when absolutely necessary. They
%should come after the References. If there is more than one
%appendix, number them alphabetically. Number displayed equations
%occurring in the Appendix in this way, e.g.~(\ref{app1}), (A.2),
%etc.
%\begin{equation}
%\mu(n, t) = {\sum^\infty_{i=1} 1(d_i < t, 
%N(d_i) = n)}{\int^t_{\sigma=0} 1(N(\sigma) = n)d\sigma}\,.
%\label{app1}
%\end{equation}

%\section*{References}
%
%References are to be listed in the order cited in the text in Arabic
%numerals.  They can be typed after punctuation marks, e.g.~``$\ldots$
%in the statement \cite{3}.'' or used directly, e.g.~``see
%[\refcite{5}] for examples.'' Please list using the style shown in
%the following examples.  For journal names, use the standard
%abbreviations.  Typeset references in 9 pt roman.

\end{document}